\documentclass[a4paper,11pt]{amsart}
\usepackage{amssymb,amsfonts,amsxtra,
mathrsfs,placeins,graphicx,verbatim,stmaryrd}
\usepackage[all]{xy}
\xyoption{line}
\usepackage{fullpage}
\usepackage{euscript}
\newtheorem{theorem}{Theorem}[section]
\newtheorem{cor}[theorem]{Corollary}
\newtheorem{lem}[theorem]{Lemma}
\newtheorem{prop}[theorem]{Proposition}

\theoremstyle{definition}

\newtheorem{example}[theorem]{Example}
\newtheorem{defi}[theorem]{Definition}
\newtheorem{rem}[theorem]{Remark}

\numberwithin{equation}{section}

\DeclareMathOperator{\Hom}{Hom}
\DeclareMathOperator{\Ass}{\textsl{ass}}
\DeclareMathOperator{\Lie}{\textsl{lie}}
\DeclareMathOperator{\MC}{MC}
\DeclareMathOperator{\CE}{CE}

\DeclareMathOperator*{\tw}{tw}
\DeclareMathOperator*{\Tw}{Tw}
\def\Comm{\textsl{comm}}
\DeclareMathOperator*{\Ext}{\operatorname{Ex}}
\newcommand{\noproof}{\begin{flushright} \ensuremath{\square}
\end{flushright}}
\def\ground{\mathbf{k}}

\def\g{\mathfrak{g}}

\def\End{\mathcal{E}}

\def\O{\mathscr{O}}
\def\P{\mathscr{P}}
\def\Q{\mathscr{Q}}
\def\Z{\mathbb{Z}}

\def\id{\operatorname{id}}
\def\Der{\operatorname{Der}}

\def\a{\mathfrak{a}}
\def\as{\mathfrak{as}}
\def\c{\mathfrak{c}}
\def\cs{\mathfrak{cs}}
\def\L{\mathscr L_\infty}
\def\A{\mathscr A_\infty}

\thanks{}

\begin{document}
\begin{abstract}
We give a general treatment of the master equation in homotopy algebras and describe the operads and formal differential geometric objects governing the corresponding algebraic structures. We show that the notion of Maurer-Cartan twisting is encoded in certain automorphisms of these universal objects.
\end{abstract}
\title[Combinatorics and formal geometry of the master equation]{Combinatorics and formal geometry of the master equation}
\author{J. Chuang  \and A.~Lazarev}
\thanks{}
\address{Centre for Mathematical Science\\City University\\London
EC1V 0HB\\UK}
\email{J.Chuang@city.ac.uk}
\address{University of Leicester\\ Department of
Mathematics\\Leicester LE1 7RH, UK.}
\email{al179@leicester.ac.uk} 

\maketitle
\tableofcontents
\section{Introduction}
Let $\g$ be a differential graded associative or Lie algebra. An odd element $\xi$ in $\g$ is called \emph{Maurer-Cartan} (or MC for short) if it satisfies the
\emph{master equation}, sometimes with the adjectives `classical' or `quantum', or the MC equation: $d\xi+\frac{1}{2}[\xi,\xi]=0$. An MC element $\xi$ allows one to \emph{twist} the differential in $\g$; the twisted differential has the form $d^\xi:=d+[\xi,?]$ and the MC condition ensures that $(d^\xi)^2=0$. The master equation and MC twisting appear in various different contexts: deformation theory, homological algebra, rational homotopy theory, differential geometry and mathematical physics (see, e.g., \cite{Get, Mer, Sta}).

There is a far-reaching generalization of the notion of an MC element and MC twisting where $\g$ is taken to be an $L_\infty$ or $A_\infty$ algebra, which could be viewed as a homotopy invariant version of a Lie or associative algebra. There are two complementary ways to treat these objects: the traditional combinatorial approach, as algebras over certain \emph{operads} and the geometric approach, as homological vector fields on (possibly noncommutative) formal supermanifolds. The geometric viewpoint goes back to Kontsevich, \cite{Kon}. In this paper we adopt both points of view.

Our aim is to describe the operadic and differential geometric structures governing MC elements and MC twisting in $A_\infty$ and $L_\infty$ algebras.
In particular, we construct certain $L_\infty$ algebras of vector fields whose MC elements are themselves $L_\infty$ and $A_\infty$ algebras together with a choice of an MC element in them; the notion of MC twisting is encoded in certain `twisting' automorphisms of these $L_\infty$ algebras. We also give give versions corresponding to cyclic $L_\infty$ and $A_\infty$ algebras. A parallel result is also obtained on the operadic side of the picture by introducing certain dg operads $\L$ and $\A$ whose algebras are $L_\infty$ and $A_\infty$ algebras together with a choice of an MC element.  Again, twisting corresponds to automorphisms in these operads.

Our next collection of results centers on the calculation of the homology of the operads $\L$ and $\A$ governing MC elements and we give a rather complete answer. In this context it is useful and illuminating to consider MC elements in $\O$ algebras where $\O$ is not necessarily an $\L$ or $A_\infty$ operad.  These MC elements are also governed by certain operads, which reduce to $\L$ and $\A$ in the case when $\O=L_\infty$ or $\O=A_\infty$ and we compute the homology of these operads. Additionally,  our approach gives a conceptual explanation of some of the
results of Willwacher \cite{Wil}. In a forthcoming paper we will apply our results to construct a higher analogue of Kontsevich's characteristic
classes of homotopy algebras with values in appropriate graph complexes. This application requires extending our results in the context of modular operads and we briefly indicate how this is done in relevant places.

The paper is organized as follows. Section 1 reviews background material on operads, operadic algebras and MC elements; it contains no new results but offers a slightly less traditional treatment. For example, we consider two versions of the $L_\infty$ operad: one with the unary operation $m_1$ and the other without, and similarly for the $A_\infty$ operad.

In Section 2 we construct the operads $\L$ and $\A$ associated to $L_\infty$ and $A_\infty$ algebras with MC elements; these operads represent the combinatorial point of view mentioned above. Section 3 is devoted to understanding the same structures from the formal geometric point of view; we construct certain big $L_\infty$ algebras of formal vector fields that capture this geometric content of the MC elements. Of independent interest here is the notion of an extension of $L_\infty$ algebras and classification of such extensions. The developed theory  allows one, among other things, to build a large supply of $L_\infty$ algebras out of usual graded Lie algebras.

In section 4 we apply the operadic and formal geometric frameworks to the MC twisting itself; it turns out that the latter is encoded by certain automorphisms of the corresponding operads and algebras of formal vector fields.

In Section 5 we consider an analogue of MC twisting in operads and introduce the functor $\O\mapsto\hat{\O}$ (the `hat-construction') which allows one to treat the internal differential on the underlying space of an algebra over an operad $\O$ as part of an operadic structure of $\hat{\O}$. Somewhat surprisingly, the operad $\hat{\O}$ turns out to be always acyclic. In Section 6 this fact is used to prove that the operads $\L$ and $\A$ are likewise acyclic. This acyclicity is caused by the presence of operations of low arity in these operads and it does not, of course, imply that they have trivial categories of algebras. We also consider certain quotients of $\A$ and $\L$ which are not acyclic.

In section 7 we relate our constructions to the results of Willwacher and give an alternative construction of the operad $\Tw\O$ whose algebras are, roughly speaking, $\O$ algebras on a dg vector spaces whose differential have been MC twisted. We compute the homology of $\Tw\O$ in the case when $\O$ is the $A_\infty$ or $L_\infty$ operad.

\subsection{Notation and conventions} We mainly adopt the conventions of our
previous paper \cite{CL}. By a \emph{vector space} we will mean a $\Z/2$-graded
 vector space (also known as super-vector space) over a field
 $\ground$ of characteristic zero.
The adjective `differential graded' will
mean `differential $\Z/2$-graded' and will be abbreviated as `dg'. Thus, a dg vector space is a pair $(V,d_V)$ where
$V$ is a $\Z/2$-graded vector space and $d_V$ is a differential on it; we will frequently omit mentioning $d_V$ in cases when
its meaning is clear from the context. While we prefer the $\Z/2$-graded framework, all our results have obvious $\Z$-graded analogues.

A
(commutative) differential graded (Lie) algebra will be abbreviated
as (c)dg(l)a.
We will often invoke the notion of a \emph{formal}
(dg) vector space; this is just an inverse limit of
finite-dimensional vector spaces. This terminology was adopted in \cite{CL} and some earlier works that the present paper relies on.\footnote{This convention is not ideal since in other contexts `formal' means `quasi-isomorphic to cohomology'. A formal vector space is
more traditionally known as a linearly compact vector space, a terminology introduced by Lefschetz \cite{Lef}; we feel that in our context it is not perfect either.} An example of a formal space is
$V^*$, the $\ground$-linear dual to a discrete vector space~$V$. A
formal vector space comes equipped with a topology and whenever we
deal with a topological vector space all linear maps from or into it will
be assumed to be continuous; thus for a formal vector space $V$ we always have $V^{**}\cong
V$, even if $V$ is not finite-dimensional.  All of our unmarked tensors are understood to be taken over
$\ground$. If $V$ is a discrete space and $W=\underleftarrow{\lim}\,{W_i}$ is
a formal space we will write $V\otimes W$ for
$\underleftarrow{\lim}\,V\otimes W_i$; thus for two discrete spaces $V$ and
$U$ we have $\Hom(V,U)\cong U\otimes V^*$.
For a $\Z/2$-graded vector space $V=V_0\oplus V_1$ the symbol $\Pi
V$ will denote the \emph{parity reversion} of $V$; thus $(\Pi
V)_0=V_1$ while $(\Pi V)_1=V_0$. The association $V\mapsto\Pi V$ is a functor; thus to a given map $f:V\to U$ between
graded spaces $V$ and $U$ there corresponds a map $\Pi f:\Pi V\to\Pi U$; to alleviate the notation and avoid possible confusion we will often
write $f$ for $\Pi f$. The symbol $S_n$ stands for the symmetric group on $n$ letters;
it is understood to act on the $n$th tensor power of a
graded vector space $V^{\otimes n}$ by permuting the tensor factors.

In this paper we freely use the language of operads and modular operads as developed
in standard sources such as \cite{GiK, GK}. An operad in this paper will always mean
a non-unital operad, i.e. a collection of vector spaces $\O(n)$ together with actions of symmetric groups
$S_n$ on each $\O(n)$ and supplied with a collection of structure maps
\[
\circ_i:\O(k)\otimes\O(n)\to\O(k+n-1),\quad i=1,\ldots, k
\]
satisfying suitable axioms. If the action of symmetric groups is omitted we obtain the notion of a \emph{non-symmetric operad}.
The conventions we adopt for $L_\infty$ and $A_\infty$ algebras lead to the consideration of twisted operads.
To avoid repeated mention of cocycle twistings, we shall use the terms cyclic operad and modular operad to mean anticyclic operad and $\mathfrak{D}_\Pi$-twisted modular operad, respectively; see \cite{GeK,GK} for details.

The \emph{endomorphism operad} of a dg vector space $V$ is the dg operad of the form $\End(V):=\{\Hom(V^{\otimes n},V)\},\, n=1,2,\ldots$. An \emph{algebra} over a dg operad $\O$ is a map of dg operads $\O~\to~\End(\Pi V)$. This slightly unorthodox definition (of course, equivalent to the usual one) allows us to minimize the sign issues present. This conflicts with the more traditional usage (e.g. an associative algebra is not an algebra over the operad of associative algebras under this definition) but since we only consider it in the context of cobar-construction type operads it should not lead to confusion. If $V$ is endowed with a symmetric inner product then, according to our convention, the operad $\End(\Pi V)$ is cyclic.

We also adopt the following convention. If $f\in\O(n)$ and $g_l\in\O(k_l)$ for $l=1,\ldots,i\leq n$ then
we denote by $f\circ(g_1\otimes\ldots \otimes g_i)$ the operation $(\ldots((f\circ_i g_i)\circ_{i-1}g_{i-1})\circ\ldots\circ_1 g_1)\in\O(k_1+\ldots+k_i+n-i)$. The following picture for $f\circ(g_1\otimes g_2\otimes g_3)$ illustrates this notation; here $f\in\O(5)$, $g_1\in\O(2)$, $g_2\in\O(3)$ and $g_3\in\O(1)$.
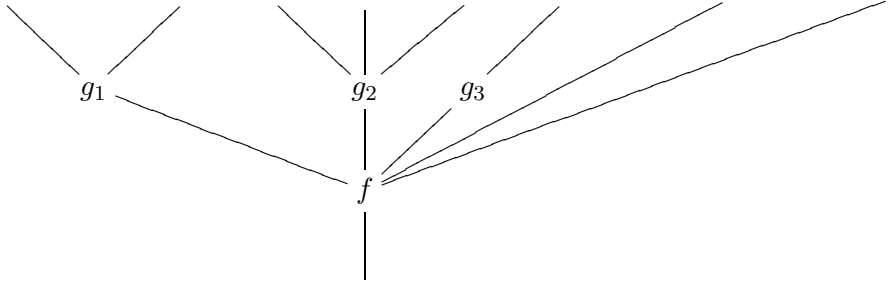
\[
\xymatrix
{
\ar@{-}[dr]&&\ar@{-}[dl]&\ar@{-}[dr]&\ar@{-}[d]&\ar@{-}[dl]&\ar@{-}[dl]&&\ar@{-}[ddllll]&&\ar@{-}[ddllllll]\\
&g_1\ar@{-}[drrr]&&&g_2\ar@{-}[d]&g_3\ar@{-}[dl]\\
&&&&f\ar@{-}[d]&&&\\
&&&&
}
\]
\subsection{Linearly topologized  dg algebras and  operads}
We often find it necessary to work with (dg) vector spaces supplied with a linear topology; in all such cases the topological dg vector spaces in question will be inverse limits of discrete dg vector spaces. For instance, an inverse limit of finite-dimensional dg vector spaces is what we call a formal dg vector space but we will see other examples as well. Correspondingly, algebraic structures on topological vector spaces have to be compatible with the topology. An instance of such an algebraic structure is a \emph{formal non-unital cdga} which is, by definition, an inverse limit of (non-unital) finite-dimensional nilpotent cdgas; note that this condition is stronger than simply being
a cdga in formal dg vector spaces. We will normally consider \emph{unital} formal cdgas (omitting the adjective `unital') that are the result of adjoining units to non-unital formal cdgas. The augmentation ideal in a formal cdga $A$ will be denoted by $A_+$.

For a dg operad $\O$ we will denote by $\O[x]$ the operad obtained from $\O$ by freely adjoining a variable $x$. The context will always make clear
the arity $n$ of the operation $x$, and it will be assumed that $S_n$ acts trivially on $x$. The notation $\O[[x]]$ will stand for the completion of $\O[x]$ at the operadic ideal $(x):\O[[x]]:=\underleftarrow{\lim}\, \O[x]/x^n$. Note that the operad $\O[[x]]$ is topological, in the sense that every dg vector space $\O(n),\, n=1,2,\ldots$ is topological and all operadic compositions are continuous maps.

Whenever we have homomorphisms between algebraic objects with linear topology (dg vector spaces, algebras or operads) these will be assumed to be continuous without mentioning this explicitly. For instance, an \emph{algebra} over a topological dg operad $\O$ is a \emph{continuous} map of operads
$\O\to\End(\Pi V)$.

Next, let $V$ be a dg vector space, $A$ be a formal cdga and $\O$ be a (possibly topological) dg operad. Then we refer to
a (continuous) map of operads $\O\to A_+\otimes\End(\Pi V)$ as an \emph{$A$-linear $\O$ algebra structure} on $V$; note that $A_+\otimes\End(\Pi V)$ is a topological operad since $A$ is a topological algebra.
\subsection{$L_\infty$ and $A_\infty$ operads and their algebras}
The $L_\infty$ and $A_\infty$ operads are usually introduced so that the operation $m_1$ is treated as a differential on the underlying space
of the corresponding algebra. In our treatment we view $m_1$ as part of the operadic structure.
\begin{defi}\
\begin{enumerate}
\item
The non-symmetric operad $A_\infty$ is freely generated by \emph{odd} operations $m_n\in A_\infty(n),\, n=1,2,\ldots$, with the usual cobar-differential
\[d(m_n)=\sum_{k=1}^{n}\sum_{i=1}^k m_k\circ_i m_{n-k+1}.\]
Note that we allow $m_1$ in the definition of $A_\infty$; thus for a vector space $V$ (without a differential) an operad map $A_\infty\to \End(\Pi V)$ is
an $A_\infty$ algebra with possibly non-zero $m_1$.

In common with any non-symmetric operad, $A_\infty$ gives rise to a usual (symmetric) operad $\{A_\infty(n)\otimes\ground[S_n]\}$. Abusing notation, we will denote the latter operad also by $A_\infty$; the context will always makes clear which of the two versions is being considered.

The operad $a_\infty$ is defined as the quotient of $A_\infty$ by the ideal $(m_1)$.
\item
The operad $L_\infty$ is a suboperad in (the symmetric version of) the operad $A_\infty$ generated by the elements $m_n^\prime:=\sum_{\sigma\in S_n}\sigma m_n$, so that the action of $\sigma\in S_n$ on $m_n^\prime$ is trivial. The differential in $L_\infty$ is given by the formula
\[d(m'_n)=\sum_{k=1}^{n}\sum_{\sigma \in \operatorname{Sh}(n-k+1,k-1)}  \sigma \left(m'_k\circ_1 m'_{n-k+1}\right),\]
where the inner sum is taken over the set of $(n-k+1,k-1)$ shuffles in $S_n$.
Note that we allow $m_1^\prime$ in the definition of $L_\infty$; thus for a vector space $V$ (without a differential) an operad map $L_\infty\to \End(\Pi V)$ is
an $L_\infty$ algebra with possibly non-zero $m_1^\prime$.

Later on we will omit the superscript $'$ and use the symbols $m_n$ to denote the $L_\infty$ as well as the $A_\infty$ operations; the context will always allow one to recover the correct meaning. Note that any $A_\infty$ algebra can be viewed as an $L_\infty$ algebra via this map;
this is an infinity version of the commutator Lie algebra associated to an associative algebra.

The operad $l_\infty$ is defined as the quotient of $L_\infty$ by the ideal $(m_1)$.
\end{enumerate}
\end{defi}
\begin{rem}\
\begin{enumerate}
\item
The operad $A_\infty$ is the cobar-construction of the operad $\Ass_e$, the unital version of the operad governing associative algebras.
Recall that the (symmetric) operad $\Ass_e$ has $\Ass_e(n)=\ground[S_n]$, $n=1,2,\ldots$. The operad $\Ass_e$ maps to the (unital) operad $\Comm_e$ for which $\Comm_e(n)=\ground$, $n=1,2,\ldots$, governing commutative algebras. It is clear from this description that the operad $L_\infty$ is the cobar-construction of $\Comm_e$ and thus, is free on the $S_n$-invariant operations $m_n^\prime\in L_\infty(n)$ that are dual to the operators in $\Comm_e(n)$ representing the $n$-fold iterated multiplication.
\item
Since $A_\infty$ and $L_\infty$ are cobar-constructions of unital operads, they are acyclic, in contrast with $a_\infty$ and $l_\infty$ (which are cobar-constructions of \emph{non-unital} $\Ass$ and $\Comm$ respectively).
This fact, however, \emph{does not} mean that the homotopy theory of algebras over the operads $A_\infty$ and $L_\infty$ is trivial. Indeed, even though the operads $A_\infty, L_\infty$ are themselves homologically trivial, they are not \emph{cofibrant} (owing to their unary part) and thus, maps out of them cannot be viewed as homotopically trivial.
\item The operads $A_\infty$ and $L_\infty$ can be obtained from $a_\infty$ and $l_\infty$ with the help of a certain general construction which will be discussed in Section \ref{secw}. The relationship between $A_\infty$ and $L_\infty$ algebras and their lower case counterparts will also be  considered in some detail there.
\end{enumerate}
\end{rem}
Specializing our general definitions of operadic algebras we arrive at the notions of  $A_\infty$ and $L_\infty$ algebras, and of $A$-linear
$A_\infty$ and $L_\infty$ algebras over a cdga.

\begin{defi}Let $V$ be a dg vector space and $A$ be a formal cdga.
\begin{enumerate}
\item
\begin{enumerate}
\item
An $L_\infty$ algebra structure on $V$ is a map of operads $L_\infty\to \End(\Pi V)$; it is thus a collection of multilinear maps
\[
m=\{m_n:(\Pi V)^{\otimes n}\to \Pi V, n=1,2,\ldots\}
\]
that are symmetric and satisfy the usual $L_\infty$ identities. The pair $(V,m)$ will be referred to as an $L_\infty$ algebra.
\item
An $A_\infty$ algebra structure on $V$ is a map of operads $A_\infty\to \End(\Pi V)$; it is thus a collection of multilinear maps
\[
m=\{m_n:(\Pi V)^{\otimes n}\to \Pi V, n=1,2,\ldots\}
\]
that satisfy the usual $A_\infty$ identities. The pair $(V,m)$ will be referred to as an $A_\infty$ algebra.
\end{enumerate}
\item
\begin{enumerate}
\item
An $A$-linear $L_\infty$ algebra structure on $V$ is a map of operads $L_\infty\to A_+\otimes\End(\Pi V)$; it is thus a collection of multilinear maps
\[
m=\{m_n:(\Pi V)^{\otimes n}\to A_+\otimes\Pi V, n=1,2,\ldots\}
\]
that are symmetric and satisfy the usual $L_\infty$ identities. The pair $(V,m)$ will be referred to as an $A$-linear $L_\infty$ algebra.
\item
An $A$-linear $A_\infty$ algebra structure on $V$ is a map of operads $A_\infty\to A_+\otimes\End(\Pi V)$; it is thus a collection of multilinear maps
\[
m=\{m_n:(\Pi V)^{\otimes n}\to A_+\otimes\Pi V, n=1,2,\ldots\}
\]
that satisfy the usual $A_\infty$ identities. The pair $(V,m)$ will be referred to as an $A$-linear $A_\infty$ algebra.
\end{enumerate}
\end{enumerate}
\end{defi}
\begin{rem}
The operation $m_1$ gives $V$ the structure of a dg vector space, which is distinct from the one associated with the internal differential on $V$.
Thus, our definition of an $L_\infty$ algebra is seemingly more general than the usual one (unless the internal differential on $V$ vanishes). On the other hand, if $m_1$ acts trivially on $V$ the given $L_\infty$ algebra structure on $V$ descends to an $l_\infty$ algebra structure recovering the usual notion.  We shall return to the question of the distinction between $L_\infty$ and $l_\infty$ algebras in Section~\ref{secw}. Similar remarks apply to $A_\infty$ and $a_\infty$ structures.
\end{rem}
One can also view an $L_\infty$ or $A_\infty$ structure on $V$ as a homological vector field on $\Pi V$; this point of view will be recalled in section \ref{formal}.  We refer for a more detailed analysis of this to
\cite{CL} and references therein.
\begin{rem}
The operads $L_\infty$ and $A_\infty$ are \emph{cyclic}, cf. \cite{GeK} concerning this notion, and every statement about them made in this subsection
has an obvious cyclic analogue. For example, a cyclic $L_\infty$ algebra is a map of cyclic operads $L_\infty\to\End(\Pi V)$ where $V$ is a vector space with a linear symplectic structure giving $\End(\Pi V)$ the structure of a cyclic operad.
\end{rem}
\subsection{Maurer-Cartan elements in $L_\infty$ and $A_\infty$ algebras}
\begin{defi}\label{defMC}Let $(V,d)$ be a dg vector space and $A$ be a formal cdga.
\begin{enumerate}
\item
Let $m$ be an $A$-linear $L_\infty$-structure on $V$. Then an even element $\xi\in A_+\otimes \Pi V$ is
\emph{Maurer-Cartan} (MC) if $(d_A\otimes\id+\id\otimes d_{\Pi V})(\xi)+\sum_{i=1}^\infty
\frac{1}{i!}m_i(\xi^{\otimes i})=0$.
\item
Let $m$ be an $A$-linear $A_\infty$-structure on $V$ and $A$. Then an even element $\xi\in A_+\otimes \Pi V$ is
\emph{Maurer-Cartan} (MC) if $(d_A\otimes\id+\id\otimes d_{\Pi V})(\xi)+\sum_{i=1}^\infty m_i(\xi^{\otimes
i})=0$.
\end{enumerate}
The set of MC elements in $A_+\otimes \Pi V$ is
denoted by $\MC(((V,d),m),A)$ but we will usually shorten it to $\MC((V,m),A)$ or even $\MC(V,A)$ leaving $d$ and $m$ understood.
\end{defi}
\begin{rem}
It would, perhaps, be notationally more natural (in the context of the above definition) to consider the set $\MC(V,B)$ where
$B$ is a formal cdga \emph{over $A$}, however we will never need this level of generality.
\end{rem}
 We refer to \cite{CL} for some basic facts on
MC elements in $L_\infty$ and $A_\infty$ algebras. In some situations it is
useful to consider MC elements with values in not necessarily formal cdgas, particularly in $\ground$, however in general
it makes no sense since the MC condition involves a possibly divergent series. Note that in the case when $V$ is a dgla or a dga
such a problem does not arise and the MC constraint takes the form of the equation of a flat connection $d(\xi)+\frac{1}{2}[\xi,\xi]=0$.
\section{Operads associated to MC elements}
Let $(V,m)$ be an $L_\infty$ or $A_\infty$ algebra and $\xi$ be an MC element in it with values in a formal cdga $A$. We will describe the operad governing this collection of data.
\begin{defi}\label{defLin}\
\begin{enumerate}
\item
Let $\L$ be the operad $L_\infty[[x]]$, where the generator $x\in\L(0)$. The differential in $L_\infty\subset\L$ is the standard cobar differential whereas
\begin{equation}\label{linfty}
d(x)=\sum_{n=1}^\infty \frac{1}{n!}m_n\circ
(\underbrace{x\otimes\ldots \otimes x}_n)
\end{equation}
\item
Let $\A$ be the non-symmetric operad $A_\infty[[x]]$, where the generator $x\in\A(0)$. The differential in $A_\infty\subset\A$ is the cobar differential whereas
\begin{equation}\label{ainfty}
d(x)=\sum_{n=1}^\infty m_n\circ(\underbrace{x\otimes\ldots \otimes x}_n)
\end{equation}
\end{enumerate}
\end{defi}
\begin{lem}
The differential $d$ in both $\L$ and $\A$ satisfies $d^2=0$.
\end{lem}
\begin{proof}
Let us treat the $\A$ case first. It suffices to prove that $d^2(x)=0$. We have, taking into account that the operations $m_n$ are odd,
\begin{align*}
d^2(x)&=d\left(\sum_{n=1}^\infty m_n\circ(x\otimes\ldots \otimes x)\right)\\
&=\sum_{n=1}^\infty d(m_n)\circ(x\otimes\ldots \otimes x)-\sum_{n=1}^\infty m_n\circ\left(d(x)\otimes x\otimes\ldots \otimes x+\ldots+x\otimes\ldots\otimes x\otimes d(x)\right).
\end{align*}
The last expression involves sums of planar trees having precisely one edge connecting vertices of valence greater than one. An example of such a tree is depicted below.
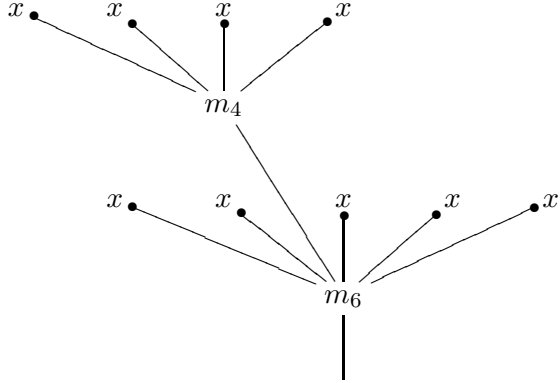
\[
\xymatrix
{
x\ar@{{*}-}[drr]&x\ar@{{*}-}[dr]&x\ar@{{*}-}[d]&x\ar@{{*}-}[dl]\\
&&m_4\ar@{-}[ddr]\\
&x\ar@{{*}-}[drr]&x\ar@{{*}-}[dr]&x\ar@{{*}-}[d]&x\ar@{{*}-}[dl]&x\ar@{{*}-}[dll]&\\
&&&m_6\ar@{-}[d]\\
&&&}
\]

Moreover every such tree contributes precisely one summand to $\sum_{n=1}^\infty d(m_n)\circ(x\otimes\ldots \otimes x)$ and one summand to $m_n\circ(d(x)\otimes x\otimes\ldots \otimes x+\ldots+x\otimes\ldots\otimes x\otimes d(x))$. These contributions cancel and we get $d^2(x)=0$.

Let us now consider the $\L$ case. Note that the canonical injection of operads $L_\infty\to A_\infty$ extends naturally to an inclusion
$L_\infty[[x]]\to A_\infty[[x]]$ and this inclusion is compatible with the operators $d$ on $L_\infty[[x]]$ and $A_\infty[[x]]$ defined by formulas (\ref{linfty}) and (\ref{ainfty}). Since $d$ squares to zero in $A_\infty[[x]]$ it follows that it also squares to zero in $L_\infty[[x]]$.
\end{proof}
The following result describes explicitly algebras over $\A$ and $\L$.
\begin{theorem}\label{twalg}
Let  $A$ be a formal cdga.
\begin{enumerate}
\item
Let $V$ be an $A$-linear $L_\infty$ algebra and $\xi\in \MC(V,A)$. Then $V$ has the structure of an $A$-linear $\L$ algebra. Conversely, an $A$-linear algebra over $\L$ gives rise to an $A$-linear $L_\infty$ algebra $V$ together with a choice of an element in $\MC(V,A)$.
\item
Let $V$ be an $A$-linear $A_\infty$ algebra and $\xi\in \MC(V,A)$. Then $V$ has the structure of a $A$-linear $\A$ algebra. Conversely, an $A$-linear algebra over $\A$ gives rise to an $A$-linear $A_\infty$ algebra $V$ together with a choice of an element in $\MC(V,A)$.
\end{enumerate}
\end{theorem}
\begin{proof}
The structure of an $A$-linear $L_\infty$ algebra over $\L$ on $V$ is the same as a continuous map of operads
\[f:\L\to A_+ \otimes \End(\Pi V).\]
Having such a map is, in turn, equivalent to having a continuous map of operads $L_\infty\to A_+ \otimes \End(\Pi V)$ (that is, an $A$-linear $L_\infty$ structure on $V$) together with an element $f(x)\in A_+\otimes \Pi V$ for which $df(x)=f(dx)$. The latter equality is precisely the condition that $f(x)$ is an MC element in $A_+\otimes \Pi V$. The $A_\infty$ case is considered similarly.
\end{proof}
\begin{rem}
Note that the operads $\L$ and $\A$ are cyclic operads in an obvious way and Theorem \ref{twalg} continues to hold in the context of cyclic $L_\infty$ and $A_\infty$ algebras.
\end{rem}
We next discuss the analogues of the above results in the context of modular operads, cf. \cite{GK}. We do not insist that our modular operads $\{\O((g,n))\}$ satisfy the stability condition $2g+n-2\leq 0$. All our conventions concerning (topological) dg operads extend in an obvious way to (topological) dg modular operads.  The cyclic operads $L_\infty$ and $A_\infty$ governing the $L_\infty$ and $A_\infty$ algebras admit modular closures $\overline{L}_\infty$ and $\overline{A}_\infty$ respectively. Then $\overline{L}_\infty((n))$ is spanned by isomorphism classes of graphs with $n$ legs and $\overline{A}_\infty((n))$ is spanned by isomorphism classes of ribbon graphs with $n$ legs. Modular algebras over $\overline{L}_\infty$ and $\overline{A}_\infty$ are the same as cyclic algebras over $L_\infty$ and $A_\infty$; therefore one can speak about MC elements in them with values in formal cdgas. The modular closures of the operads $\L$ and $\A$ will be denoted by $\overline{\mathscr{L}}_\infty$  and $\overline{\mathscr{A}}_\infty$  respectively.

We then have the following analogue of Theorem \ref{twalg}.
\begin{theorem}
Let  $A$ be a formal cdga.
\begin{enumerate}
\item
Let $V$ be an $A$-linear $\overline{L}_\infty$ algebra and $\xi\in \MC(V,A)$. Then $V$ has the structure of an $A$-linear $\overline{\mathscr{L}}_\infty$ algebra. Conversely, an $A$-linear algebra over $\overline{\mathscr{L}}_\infty$  gives rise to an $A$-linear $\overline{L}_\infty$ algebra $V$ together with a choice of an element in $\MC(V,A)$.
\item
Let $V$ be an $A$-linear $\overline{A}_\infty$ algebra and $\xi\in \MC(V,A)$. Then $V$ has the structure of an $A$-linear $\overline{\mathscr{A}}_\infty$ algebra. Conversely, an $A$-linear  algebra over $\overline{\mathscr{A}}_\infty$  gives rise to an $A$-linear $\overline{A}_\infty$ algebra $V$ together with a choice of an element in $\MC(V,A)$.
\end{enumerate}
\end{theorem}
\noproof
\section{$L_\infty$ algebras of formal vector fields}\label{formal}
The purpose of this section is to describe $\A$ and $\L$ algebras in more geometric terms. Recall that $A_\infty$ and $L_\infty$ algebras are MC elements in certain (infinite dimensional) Lie algebras. It turns out that analogous structures governing $\A$ and $\L$ algebras are themselves $L_\infty$ algebras.
\subsection{Extensions of $L_\infty$ algebras} In order to introduce $L_\infty$ algebras of formal vector fields in a conceptual way we need the notion of an \emph{extension} of $L_\infty$ algebras which is of independent interest. First recall the geometric definition of an $L_\infty$ algebra, and of an $A_\infty$ algebra which we also have the opportunity to use later on.
\begin{defi}\
\begin{enumerate}
\item
Let $(V,d_V)$ be a dg vector space. An $L_\infty$ structure on $V$ is an odd element $m~\in~\Der(\hat{S}\Pi V^*)$ which has no constant term and satisfies the master equation $d_V(m)+\frac{1}{2}[m,m]~=~0$. The pair $(V,m)$ will be referred to as an $L_\infty$ algebra and the algebra $\hat{S}\Pi V^*$, supplied with the differential $d_V+m$, as its representing formal cdga. The internal differential $d_V$ will often be suppressed from the notation.
For two $L_\infty$ algebras $(V,m_V)$ an $(U,m_U)$ and $L_\infty$ map $V\to U$ is, by definition, a continuous map of its representing cdgas $\hat{S}\Pi U^*\to \hat{S}\Pi V^*$.
\item
Let $(V,d_V)$ be a dg vector space. An $A_\infty$ structure on $V$ is an odd element $m~\in~\Der(\hat{T}\Pi V^*)$ which has no constant term and satisfies the master equation $d_V(m)+\frac{1}{2}[m,m]~=~0$. The pair $(V,m)$ will be referred to as an $A_\infty$ algebra and the algebra $\hat{T}\Pi V^*$, supplied with the differential $d_V+m$, as its representing formal dga. The internal differential $d_V$ will often be suppressed from the notation.
For two $A_\infty$ algebras $(V,m_V)$ and $(U,m_U)$ an $A_\infty$ map $V\to U$ is, by definition, a continuous map of its representing dgas $\hat{T}\Pi U^*\to \hat{T}\Pi V^*$.
\end{enumerate}
\end{defi}
\begin{rem}We would like to make several comments on the above definition of $A_\infty$ and $L_\infty$ structures; to avoid repeating ourselves we confine those comments to the $L_\infty$ context, however it is clear that they also apply with appropriate modifications to the $A_\infty$ case.
\label{FvcRemarks}
\begin{enumerate}
\item
The above definition of an $L_\infty$ structure is, of course, equivalent to the one via the operad $L_\infty$ (and so, generalizes the traditional one to which it reduces when the internal differential $d_V$ on $V$ vanishes). The derivation $m:\hat{S}\Pi V^*\to \hat{S}\Pi V^*$ determining an $L_\infty$ structure on $V$ has the form $m=m_1+m_2+\ldots$ where $m_n,n=1,2,\ldots$ is the component of $m$ landing in $\hat{S}^n\Pi V^*$, so that $m_n:\Pi V^*\to (\Pi V^*)^{\otimes n}$. By abuse of notation we have the same symbol for this map as for the structure map $(\Pi V)^{\otimes n}\to \Pi V$ (to which it is dual). The notion of an $L_\infty$ morphism does not have a straightforward interpretation in terms of operads.
\item
Note that an $L_\infty$ structure, as defined above, is an \emph{odd} derivation. It follows that if one adopts our convention to view an MC element in a dgla $\g$ as an \emph{even} element in $\Pi \g$ (which ameliorates the sign issues) then one cannot consider an $L_\infty$ structure as an MC element. It is, of course, possible to recover consistency by applying the parity shift to the Lie algebra of formal vector fields but this creates other notational complications which do not justify that. This discrepancy underscores the different roles played by the `big' $L_\infty$ algebras (such as Lie algebras of formal vector fields) and `small' $L_\infty$ algebras which are themselves MC elements in those big Lie algebras.
\item
Just as in the operadic context the above definition admits an extension covering $A$-linear $L_\infty$ algebras where $A$ is a formal cdga. Namely, an $A$-linear $L_\infty$ algebra on $V$ is a continuous $A$-linear odd derivation of $A\otimes \hat{S}\Pi V^*$ having its image in $A_+\otimes \hat{S}\Pi V^*$,  which satisfies the master equation, and has no constant term; the notion of an $A$-linear $L_\infty$ algebra is defined similarly.
\item If $\hat{S}\Pi U^*\to \hat{S}\Pi V^*$ is a map of formal cdgas representing an $L_\infty$ map $V\to U$ then $f$ can be written as $f=f_1+f_2
+\ldots$ where $f_n:\Pi U^*\to(\Pi V^*)^{\otimes n}, n=1,2,\ldots$ is the component of $f$ of degree $n-1$ with respect to the tensor grading on $\hat{S}\Pi U^*$ and $\hat{S}\Pi V^*$. Note that the component $f_0$ is absent since there are no nontrivial algebra maps $\hat{S}\Pi U^*\to\ground$. However if $U$ and $V$ are $A$-linear $L_\infty$ algebras for a formal cdga $A$ then an $A$-linear $L_\infty$ map represented by a map of cdgas $f:\hat{S}\Pi U^*\to A\otimes\hat{S}\Pi V^*$ \emph{might} have a nontrivial $f_0$-part (corresponding to a cdga map $\hat{S}\Pi U^*\to A$. An $L_\infty$ map $f$ having nonzero $f_0$ will be referred to as a \emph{curved} $L_\infty$ map.
\end{enumerate}
\end{rem}
\begin{defi}\
\begin{enumerate}
\item
Let $(V,m)$ be an $L_\infty$ algebra; then $I\subset V$ is called an $L_\infty$ ideal in $V$ if for any $n$ we have $m_n(\Pi v_1,\ldots,\Pi v_n)\in\Pi I$ as long as at least one of the $v_i$'s belongs to $I$.
It is clear that in that case both $I$ and $U=V/I$ inherit structures of $L_\infty$ algebras.
\item If, in addition, the sequence of $L_\infty$ algebras
\[
I\to V\to U
\]
splits as a sequence of dg vector spaces then it is called an \emph{extension} of $U$ by $I$; the choice of a splitting is thus part of the data of an extension. The set of such extensions will be denoted by $\Ext_{L_\infty}(U,I)$.
\end{enumerate}
\end{defi}
\begin{rem}
It is useful to reformulate the notion of an extension of $L_\infty$ algebras in terms of its representing formal cdgas. Given an extension $e:I\to V\to U$ consider the formal cdgas $\hat{S}\Pi V^*$ and $\hat{S}\Pi U^*$ representing the $L_\infty$ algebras $V$ and $U$. Then the condition that $I$ is an $L_\infty$ ideal in $V$ is equivalent to saying that $\hat{S}\Pi U^*$ is a sub cdga in $\hat{S}\Pi V^*$.
\end{rem}
The set $\Ext_{L_\infty}(U,I)$ is contravariantly functorial with respect to $L_\infty$ maps in the variable $U$. Indeed, suppose that $e:I\to V\to U$ is an extension and let an $L_\infty$ map
$W\to U$ be represented by a map of formal cgdas $f:\hat{S}\Pi U^*\to\hat{S}\Pi W^*$ and consider the sequence of maps of formal cdgas
\begin{equation}\label{ext}
\hat{S}\Pi W^*\to \hat{S}\Pi W^*\otimes_{\hat{S}\Pi U^*}\hat{S}\Pi V^*\to\hat{S}\Pi I^*.
\end{equation}
Here $\hat{S}\Pi W^*$ is regarded as a ${\hat{S}\Pi U^*}$-module via $f$. Note that (\ref{ext}) represents an $L_\infty$ extension of $W$ by $I$. To see that note that there is a decomposition $V\cong I\oplus U$ associated to the extension $e$ and it follows that the representing cdga of $V$ then has the form $\hat{S}\Pi V^*\cong\hat{S}\Pi I^*\otimes\hat{S}\Pi U^*$. Thus, the sequence (\ref{ext}) can be rewritten as
\[
\hat{S}\Pi W^*\to \hat{S}\Pi I^*\otimes\hat{S}\Pi W^* \to\hat{S}\Pi I^*
\]
which clearly does represent an $L_\infty$ extension of $W$ by $I$. This extension will be denoted by $f^*(e)$.

The functor $U\mapsto\Ext_{L_\infty}(U,I)$ is analogous to the functor associating to a topological space the set of fibrations with a typical fiber having a fixed homotopy type (the $L_\infty$ algebra $I$ playing the role of the said fiber). The extension $f^*(e)$ described above should be viewed as an analogue of the induced fibration.  In light of this analogy the following result is not unexpected.
\begin{prop}\label{extMC}
The functor $U\mapsto \Ext_{L_\infty}(U,I)$ is represented by the dgla $\Der(\hat{S}\Pi I^*)$. In more detail: there is natural bijective correspondence between
the elements in $\Ext_{L_\infty}(U,I)$ and $L_\infty$ maps $U\to\Der(\hat{S}\Pi I^*)$.
\end{prop}
\begin{proof}
Let $I\to V\to U$ be an extension; its dg splitting gives an isomorphism of dg vector spaces $V\cong I\oplus U$. We have the following sequence of representing formal cdgas: $\hat{S}\Pi U^*\to\hat{S}\Pi U^*\otimes\hat{S}\Pi I^*\to\hat{S}\Pi I^*$. The $L_\infty$ structure $m_V$ on $V=\hat{S}\Pi U^*\otimes\hat{S}\Pi I^*$ is determined by its restriction onto $\hat{S}\Pi I^*$. This restriction can be viewed as an $\hat{S}\Pi U^*$-linear odd derivation of $\hat{S}\Pi U^*\otimes\hat{S}\Pi I^*$, i.e. an element $\xi\in\hat{S}\Pi U^*\otimes\Der(\hat{S}\Pi I^*)$. The condition $m_V^2=0$ gets translated into the MC condition for $\xi$. This establishes a 1-1 correspondence between $L_\infty$ extensions of $U$ by $I$ and elements in $\MC(\Der(\hat{S}\Pi I^*), \hat{S}\Pi U^*)$. It only remains to observe that an $L_\infty$ map $U\to \Der(\hat{S}\Pi I^*)$ is also naturally identified with an element in $\MC(\Der(\hat{S}\Pi I^*), \hat{S}\Pi U^*)$.
\end{proof}
Let us see in more concrete terms how an $L_\infty$ map $U\to \Der(\hat{S}\Pi I^*)$ gives rise to an $L_\infty$ structure on $V=U\oplus I$. The latter is specified by the collection of structure maps $m_n:(\Pi V)^{\otimes n}\to V$. Since $m_n$ is symmetric it suffices to specify its value on the collection
$(\Pi u_1,\ldots,\Pi u_k,\Pi x_1,\ldots,\Pi x_l)$ where $u_1,\ldots, u_k\in U, x_1,\ldots,x_l\in I$ and $k+l=n$.

 The proof of Proposition \ref{extMC} yields the following corollary.
Here we adopt the convention of Remark~\ref{FvcRemarks}(1):
a derivation $\eta\in \Der(\hat{S}\Pi I^*)$
is determined by a collection of maps
$\Pi I^*\to (\Pi I^*)^{\otimes l}$, $l=0,1,2,\ldots$,
and we denote by $\eta_l:
(\Pi I)^{\otimes l}\to \Pi I$ the dual maps.

\begin{cor}
Let $f=(f_1,f_2,\ldots) $ be an $L_\infty$ map  $U\to \Der(\hat{S}\Pi I^*)$. Then the $L_\infty$ structure $m^f=(m_1^f,m_2^f,\ldots)$ on $U\oplus I$ associated with $f$ has the form:

\[
m_n^f(w_1,\ldots,w_k,x_1,\ldots,x_l)=(\Pi f_k(w_1,\ldots,w_k))_l(x_1,\ldots,x_l)
\]
where $w_1,\ldots, w_k\in \Pi U$, $x_1,\ldots,x_l\in I$, $1\leq l\leq n$, $0\leq k\leq n-1$ and $k+l=n$; the value of $m_n^f$ on other collections of arguments is zero. In the case $l=0,\, k=n$ the formula is modified as follows:

\[
m_n^f(w_1,\ldots, w_n)=(\Pi f_n(w_1,\ldots,w_n))_0+m^U_n(w_1,\ldots, w_n)
\]
where $m^U_n$ the is corresponding $L_\infty$ structure map for $U$.
\end{cor}
\noproof
\begin{rem}\
\begin{enumerate}
\item
Note that the dgla $\Der(\hat{S}\Pi I^*)$ is just the Chevalley-Eilenberg complex $\CE(I,I)$ of the $L_\infty$ algebra $I$ with coefficients in itself, cf. \cite{CL}.
\item
The classical notion of an extension of a Lie algebra by an ideal is, of course, a (very) special case of an $L_\infty$ extension. Indeed, suppose that $I\to V\to U$ is such an extension. It corresponds, via Proposition \ref{extMC} to an $L_\infty$ map $f=f_1+f_2+\ldots:U\to \Der(\hat{S}\Pi I^*)$ for which, firstly, all higher components $f_3,f_4,\ldots$ are zero, secondly, the image of $f_1:U\to \Der(\hat{S}\Pi I^*)$ consists of linear derivations $\operatorname{End}(\Pi I^*)\in\Der(\hat{S}\Pi I^*)$ and thirdly, the image of $f_2:U\otimes U\to \Der(\hat{S}\Pi I^*)$ consists of constant derivations $\Pi I^*\in\Der(\hat{S}\Pi I^*)$.  The subspace consisting of linear and constant derivations forms a Lie subalgebra in $ \Der(\hat{S}\Pi I^*)$ but it is not closed with respect to the differential unless the Lie bracket in $I$ is zero. In the latter case $f_1$ gives $I$ the structure of a $U$-module and $f_2$ determines a 2-cocycle of $U$ with values in $I$.  Extensions of this kind are classified, as expected, in terms of the second Chevalley-Eilenberg cohomology of $U$ with coefficients in $I$ but we will not elaborate this point further.
\end{enumerate}
\end{rem}
\subsection{Extensions related to Lie algebras or formal vector fields and MC elements}
The following example of an $L_\infty$ extension is what prompted us to develop the general theory in the first place.
\begin{example}\label{ex1}
Let $V$ be a dg  vector space, and $\c(V)$ be the dgla of continuous derivations of $\hat{S}\Pi V^*$ vanishing at zero. Let us consider the tautological embedding $f:\c(V)\hookrightarrow\Der(\hat{S}\Pi V^*)$; note that this is a usual map between two dglas. Associated to $f$ is an $L_\infty$ algebra structure on $\tilde{\c}(V):=\c(V)\oplus V$ defined as follows.

On the subspace $\Pi\c(V)\subset\Pi\tilde{c}(V)$ the product $m_2$ is just (the parity shift of) the bracket on the dgla $\c(V)$ and the higher $L_\infty$ products are zero. On the subspace $V$ the $L_\infty$ structure is zero. The remaining non-zero $L_\infty$ products are of the form $m_n(\Pi g,w_1,\ldots,w_{n-1})$ with $g\in\Hom((\Pi V)^{\otimes n-1},\Pi V)\subset \g(V)$ and $w_i\in \Pi V, i=1,2,\ldots, n-1$. Namely, $m_n(\Pi g, w_1,\ldots, w_{n-1})=g(w_1,\ldots,w_{n-1}).$
\end{example}
Varying the choice of a Lie subalgebra in $\Der(\hat{S}\Pi V^*)$ one can construct many other examples of nontrivial $L_\infty$ structures. We list some natural examples; the $L_\infty$ structure maps are given by the same formulas as in Example \ref{ex1}.
\begin{example}\
\begin{enumerate}
\item
Associated to the identity map $\Der(\hat{S}\Pi V^*)\to \Der(\hat{S}\Pi V^*)$ is an $L_\infty$ structure on $\Der(\hat{S}\Pi V^*)\oplus V$. This is, perhaps the most natural special case; however it will not play a role in the present paper.
\item Suppose that the dg vector space $V$ is endowed with a symmetric non-degenerate scalar product; this is equivalent to having a linear (super)symplectic structure on $\Pi V$. Denote by $\cs(V)$ the dgla of \emph{symplectic derivations} of $\hat{S}\Pi V^*$ vanishing at zero. Then the embedding of dglas $\cs(V)\hookrightarrow\Der(\hat{S}\Pi V^*)$ gives an $L_\infty$ structure on $\widetilde{\cs}(V):=\cs(V)\oplus V$.
\item Consider the dgla $\a(V)$ consisting of continuous derivations of $\hat{T}\Pi V^*$ vanishing at zero; since a derivation preserves the commutator ideal
any such derivation gives rise to a derivation of the corresponding dgla $\Der(\hat{S}\Pi V^*)$. This determines a dgla map $\a(V)\to \Der(\hat{S}\Pi V^*)$ and thus it gives rise to an $L_\infty$ structure on $\tilde{a}(V):=\a(V)\oplus V$.
\item Suppose that the dg vector space $V$ is endowed with a symmetric non-degenerate scalar product and denote by $\as(V)$ the corresponding
 dgla of symplectic derivations of $\hat{T}\Pi V^*$ (cf. for example \cite{HL'} concerning this notion). There is a dgla map $\as(V)\to \Der(\hat{S}\Pi V^*)$ defined as the composition $\as(V)\hookrightarrow \a(V)\to \Der(\hat{S}\Pi V^*)$. We will denote the corresponding $L_\infty$ structure on $\as(V)\oplus V$ by $\widetilde{\as}(V)$.
\end{enumerate}
\end{example}
For us the dglas $\c(V), \cs(V),\a(V)$ and $\cs(V)$ are important because MC elements in these are (by definition), respectively, $L_\infty$, cyclic $L_\infty$, $A_\infty$ and cyclic $A_\infty$ structures on the dg vector space $V$.
It turns out that the MC elements in the corresponding $L_\infty$ extensions of these Lie algebras also admit a very natural description.
\begin{theorem} Let $A$ be a formal cgda and $V$ is a dg vector space.
\begin{enumerate}
\item
An element in $\MC(\tilde{\c}(V),A)$ is an A-linear $L_\infty$ structure on $V$ together with a choice of an element in $\MC(V,A)$.
\item
An element in $\MC(\tilde{\cs}(V),A)$ is an A-linear cyclic $L_\infty$ structure on $V$ together with a choice of an element in $\MC(V,A)$.
\item
An element in $\MC(\tilde{\a}(V),A)$ is an A-linear $A_\infty$ structure on $V$  together with a choice of an element in $\MC(V,A)$.
\item
An element in $\MC(\tilde{\as}(V),A)$ is an A-linear cyclic $A_\infty$ structure on $V$ together with a choice of an element in $\MC(V,A)$.
\end{enumerate}
\end{theorem}
\begin{proof}
The proofs in all four cases are completely analogous and so we will restrict ourselves to   considering part (1). Let $g+\xi\in (A_+\otimes\Pi\c(V))\oplus (A_+\otimes\Pi V)=A_+\otimes\Pi\tilde{\c}(V)$ be an element in $\MC(\c(V),A)$. Taking into account the formulas for the $L_\infty$ structure maps in $\tilde{\c}(V)$ given in Example \ref{ex1} we see that the master equation for $g+\xi$,
\[
d(g+\xi)+m_1(g+\xi)+\frac{1}{2!}m_2(g+\xi,g+\xi)+\ldots+\frac{1}{n!}m_n(g+\xi,\ldots,g+\xi)+\ldots=0,
\]
is equivalent to two equations:
\[
d(g)+\frac{1}{2!}m_2(g,g)=0;
\]
\[
d(\xi)+m_2(g,\xi)+\frac{1}{2!}m_3(g,\xi,\xi)+\ldots+\frac{1}{(n-1)!}m_n(g,\xi,\ldots,\xi)+\ldots =0.
\]
The first of these equations is precisely the condition that $g=(g_1,g_2,\ldots)\in A_+\otimes\c(V)$ is an $A$-linear $L_\infty$ structure on $V$ (here $g_i\in \Hom(V^{\otimes i},A_+\otimes V$), whereas the second can be rewritten as
\[
d(\xi)+g_1(\xi)+\frac{1}{2!}g_2(\xi,\xi)+\ldots+\frac{1}{(n-1)!}g_{n-1}(\xi,\ldots,\xi) = 0,
\]
which is precisely the condition specifying $\xi$ as an MC element in the $L_\infty$ algebra with underlying space $A_+\otimes V$ whose $L_\infty$ structure is given by the collection $g=(g_1,g_2,\ldots)$.
\end{proof}
\section{Twisting in $L_\infty$ and $A_\infty$ algebras} Let us recall the standard definition of MC twisting following \cite{CL}. Our usual convention is to let the underlying graded vector space $V$ carry its own differential $d_V$; this results in a slight generalization of the treatment in op.cit.

\begin{defi}\label{closedtw}\
\begin{enumerate}
\item
Let $A$ be a formal cdga, $((V,d_V),m)$ be an $A$-linear $L_\infty$ algebra structure on a dg vector space $(V,d_V)$ and $\xi\in\MC(V,A)$, viewed as an even $A$-linear derivation of $A\otimes\hat{S}\Pi V^*$. Then $m^\xi:=e^\xi me^{-\xi}
-d_{A \otimes \Pi V}(\xi)$, considered as an odd $A$-linear derivation of $A\otimes \hat{S}\Pi V^*$, satisfies the master equation and is, therefore, an $A$-linear $L_\infty$ structure on $(V,d_V)$. This structure is called the \emph{twisted} $L_\infty$ structure (by the MC element $\xi$).
\item
Let $A$ be a formal cdga,  $((V,d_V),m)$ be an $A$-linear $A_\infty$ algebra structure on a dg vector space $(V,d_V)$, and $\xi\in\MC(V,A)$, viewed as an even $A$-linear derivation of $A\otimes\hat{T}\Pi V^*$. Then $m^\xi:=e^\xi me^{-\xi}-d_{A\otimes\Pi V}(\xi)$, considered as an odd $A$-linear derivation of $A\otimes \hat{T}\Pi V^*$ satisfies the master equation and is, therefore, an $A$-linear $A_\infty$ structure on $(V,d_V)$. This structure is called the \emph{twisted} $A_\infty$ structure (by the MC element $\xi$).
\end{enumerate}
\end{defi}\label{deftw}
We now give explicit formulas for the structure maps of twisted $L_\infty$ or $A_\infty$ algebras. As usual, we denote by $A$ a fixed formal cdga.
\begin{enumerate}
\item Let $V$ be an $L_\infty$ algebra and $\xi\in \MC(V,A)$. Then for $x_1,\ldots, x_n\in\Pi V,$
\begin{equation}\label{twist1}
m^\xi_n(x_1, \ldots, x_n)=\sum_{i=0}^\infty
\frac{1}{i!}m_{n+i}(\xi,\ldots, \xi, x_1,\ldots, x_n).
\end{equation}
\item Let $V$ be an $A_\infty$ algebra and $\xi\in \MC(V,A)$. Then for $x_1,\ldots, x_n\in\Pi V,$
\begin{equation}\label{twist2}
m^\xi_n(x_1, \ldots, x_n)=\sum_{i=0}^\infty m_{n+i}(\xi,\ldots,
\xi| x_1,\ldots, x_n),
\end{equation}
where
$$m_{n+i}(\xi,\ldots, \xi| x_1,\ldots, x_n):=\sum m_{n+i}(z_1,\ldots,z_{n+i}),$$
the sum running over
all ${n+i}\choose{n}$ sequences $z_1,\ldots,z_{n+i}$ containing $x_1,\ldots, x_n$ in order, together
with $i$ copies of  $\xi$.
\end{enumerate}
\begin{rem}

It is easy to see that $\eta\mapsto\eta-\xi$ establishes a one-to-one correspondence $\MC((V,m),A)\to\MC((V,m^\xi),A)$. In particular, $-\xi\in\MC((V,m^\xi),A)$.

\end{rem}
\subsection{Non-formal twisting in $L_\infty$ and $A_\infty$ algebras}\label{S:restricted}
We would like to outline briefly another approach to twisting in $L_\infty$ and $A_\infty$ algebras, which makes no use of a formal `coefficient' cdga. As far as we know this approach is new and it is of independent interest; however it plays a very minor role in the present paper and so we will not give too detailed a treatment.

Let $(V,d_V)$ be a dg vector space and $((V,d_V),m)$ be a corresponding $L_\infty$ structure or $A_\infty$ structure. Let $\xi\in\Pi V$ and consider, in the $L_\infty$ case, the master equation
\begin{equation}\label{openMC}
d_V(\xi)+\sum_{i=1}^\infty
\frac{1}{i!}m_i(\xi^{\otimes i})=0,
\end{equation}
which in the $A_\infty$ case should be replaced with
\begin{equation}\label{openMCA}
d_V(\xi)+\sum_{i=1}^\infty
m_i(\xi^{\otimes i})=0.
\end{equation}
Note that equations (\ref{openMC}) and (\ref{openMCA}) only make sense if the sums appearing in them are finite; a solution of either will be called an MC element in $V$ and the set of MC elements will be denoted by $\MC(V)$.  We will now describe some natural conditions on $V$ which are necessary for the existence of MC elements and for the notion of MC twisting. To this end consider any line (i.e. a one-dimensional subspace) $l\subset\Pi V$ and the ideal $I=I_l\subset S\Pi V^*$ generated by the annihilator of $l$ in $V^*$. For example, if we choose coordinates $x_i$ in $\Pi V$ so that $S\Pi V^*\cong \ground[x_1,x_2,\ldots]$ and $l$ consists of
the elements $(\alpha,0,0,\ldots)$, $\alpha\in\ground$, then the ideal $I$ is generated by $x_2,x_3,\ldots$. We denote by $\hat{S}_I\Pi V^*$ the completion of $S\Pi V^*$ at the ideal $I$ and by $\Der(\hat{S}_I \Pi V^*)$ the corresponding dgla of continuous derivations.

In the $A_\infty$ case we similarly consider the two-sided ideal $I=I_l$ in $T\Pi V^*$ generated by the annihilator of $l$ in $V^*$, the corresponding completion $\hat{T}_I \Pi V^*$ and the dgla $\Der(\hat{T}_I \Pi V^*)$. Note in passing that the dga $\hat{T}_I\Pi V^*$ and the cdga $\hat{S}_I \Pi V^*$ are both topological, but \emph{not} formal, owing to the fact that the ideal $I$ is not maximal. Because of that, nontrivial dga maps $\hat{T}_I \Pi V^*\to \ground$ and $\hat{S}_I\Pi V^*\to \ground$ may exist.
\begin{defi}\
\begin{enumerate}
\item
A \emph{restricted} $L_\infty$ structure on $V$ is an odd element $m\in\Der(\hat{S}_I\Pi V^*)$ which has no constant term and satisfies the master equation $d_V(m)+\frac{1}{2}[m,m]=0$. The triple $(V,l,m)$ will be referred to as a restricted $L_\infty$ algebra and the algebra $\hat{S}_I\Pi V^*$, supplied with the differential $d_V+m$, as its representing cdga.
\item
A \emph{restricted} $A_\infty$ structure on $V$ is an odd element $m\in\Der(\hat{T}_I\Pi V^*)$ which has no constant term and satisfies the master equation $d_V(m)+\frac{1}{2}[m,m]=0$. The triple $(V,l,m)$ will be referred to as a restricted $A_\infty$ algebra and the algebra $\hat{T}_I\Pi V^*$, supplied with the differential $d_V+m$, as its representing dga.
\end{enumerate}
\end{defi}
Then we have the following standard result whose proof is virtually identical to the proof of Proposition 2.2 of \cite{CL}.
\begin{prop}\
\begin{enumerate}
\item
Let $(V,l,m)$ be a restricted $L_\infty$ algebra. Then the set of solutions of (\ref{openMC}) with $\xi\in l$ is in 1-1 correspondence with (continuous) cdga maps
$\hat{S}_I\Pi V^*\to\ground$.
\item
Let $(V,l,m)$ be a restricted $A_\infty$ algebra. Then the set of solutions of (\ref{openMC}) with $\xi\in l$ is in 1-1 correspondence with (continuous) dga maps
$\hat{T}_I\Pi V^*\to\ground$.
\end{enumerate}
\end{prop}
\noproof
Given an MC element $\xi$ in a restricted $L_\infty$ or $A_\infty$ algebra $V$, the latter can be twisted by $\xi$.
\begin{defi}
\label{opentw}\
\begin{enumerate}
\item
Let  $(V,l,m)$ be a restricted $L_\infty$ algebra structure on a dg vector space $V$ and $\xi\in l$ be an MC-element in $V$, viewed as an even derivation of $\hat{S}_I\Pi V^*$. Then $m^\xi:=e^\xi me^{-\xi}
-d_{\Pi V}(\xi)$ considered as an odd derivation of $\hat{S}_I\Pi V^*$, satisfies the master equation and therefore determines a (restricted) $L_\infty$ structure on $V$. This structure is called the \emph{twisted} $L_\infty$ structure (by the MC element $\xi$).
\item
Let  $(V,l,m)$ be a restricted $A_\infty$ algebra structure on a dg vector space $V$ and $\xi\in l$ be an MC-element in $V$, viewed as an even derivation of $\hat{T}_I\Pi V^*$. Then $m^\xi:=e^\xi me^{-\xi}
-d_{\Pi V}(\xi)$ considered as an odd derivation of $\hat{T}_I\Pi V^*$, satisfies the master equation and therefore determines a (restricted) $L_\infty$ structure on $V$. This structure is called the \emph{twisted} $A_\infty$ structure (by the MC element $\xi$).
\end{enumerate}
\end{defi}
It is easy to see that formulas (\ref{twist1}) and (\ref{twist2}) still hold in the restricted context.
\subsection{Twisting automorphisms of Lie algebras of formal vector fields} Twisting in infinity-algebras by MC elements naturally gives rise to $L_\infty$ automorphisms in the $L_\infty$ algebras governing the corresponding structures.
\begin{theorem}Let $V$ be a dg vector space.
\begin{enumerate}
\item
Consider the automorphism $\tw$ of the functor $A\to\MC(\tilde{\c}(V),A)$ given by $\tw:(m,\xi)\mapsto (m^\xi,-\xi)$ where $A$ is a formal cdga. This automorphism induces an $L_\infty$ automorphism of $\tilde{\c}(V)$  (which will be denoted by the same symbol), with components $\tw_n:\left(\Pi\tilde{\c}(V)\right)^{\otimes n}\to \Pi\tilde{\c}(V), n=1,2,\ldots$.
For any $f\in \Hom(S^{k+n-1}\Pi V,\Pi V)$ with $k\geq 1$
and $w_1,\ldots,w_{n-1}\in \Pi V$, we have
$${\tw}_n(\Pi f, w_1,\ldots, w_{n-1})
\in
 \Pi\Hom(S^k\Pi V,\Pi V) \subset \Pi\tilde{\c}(V),$$
given by the formula
\begin{equation}\label{twc}
{\tw}_n(\Pi f, w_1,\ldots, w_{n-1})(x_1,\ldots, x_k)=\Pi f(w_1,\ldots, w_{n-1}, x_1,\ldots, x_k).
\end{equation}
 In addition $\tw_1(w)=-w$ for all $w\in \Pi V$. The value of $\tw_n$ on any other  collection of homogeneous elements is zero.

\item
Consider the automorphism $\tw$ of the functor $A\to\MC(\cs(V),A)$ given by $\tw:(m,\xi)\mapsto (m^\xi,-\xi)$. This automorphism induces an $L_\infty$ automorphism of $\tilde{\cs}(V)$  which is simply the restriction of the corresponding $L_\infty$ automorphism in part (1).
\item
Consider the automorphism $\tw$ of the functor $A\to\MC(\tilde{\a}(V),A)$ given by $\tw:(m,\xi)\mapsto (m^\xi,-\xi)$. This automorphism induces an $L_\infty$ automorphism of $\tilde{\a}(V)$  (which will be denoted by the same symbol), with components $\tw_n:\left(\Pi\tilde{\a}(V)\right)^{\otimes n}\to \Pi\tilde{\a}(V), n=1,2,\ldots$.
For any $f\in \Hom(\left(\Pi V\right)^{\otimes n-1+k},\Pi V)$ with $k\geq 1$
and $w_1,\ldots,w_{n-1}\in \Pi V$, we have
$${\tw}_n(\Pi f, w_1,\ldots, w_{n-1})
\in
 \Pi\Hom(\left(\Pi V\right)^{\otimes k},\Pi V) \subset \Pi\tilde{\a}(V),$$
given by the formula
\begin{equation}\label{twa}
{\tw}_n(\Pi f,w_1,\ldots,w_{n-1})(x_1,\ldots x_{k-n+1})
=\sum_{\sigma\in S_{n-1}} \Pi f\left(\sigma(w_1\otimes \ldots \otimes w_{n-1})\mid x_1,\ldots,x_{k-n+1}\right).
\end{equation}
  In addition $\tw_1(w)=-w$ for all $w\in \Pi V$. The value of $\tw_n$ on any other  collection of homogeneous elements is zero.
\item
Consider the automorphism $\tw$ of the functor $A\to\MC(\as(V),A)$ given by $\tw:(m,\xi)\mapsto (m^\xi,-\xi)$. This automorphism induces an $L_\infty$ automorphism of $\tilde{\as}(V)$  which is simply the restriction of the corresponding $L_\infty$ automorphism in part (3).
\end{enumerate}
\end{theorem}
\begin{proof}
Parts (2) and (4) are immediate consequences of (1) and (3) respectively. Let us prove (1). Denote temporarily by $\tw^\prime$ the automorphism of $\hat{S}\Pi [\tilde{\c}(V)]^*$ given by formula (\ref{twc}). It is sufficient to prove that for any formal cdga $A$ the automorphism $\tw^\prime$ induces the same map on $A$-linear MC elements $(m,\xi)\in\MC(\tilde{\c}(V),A)$ as $\tw$ does. It would follow then that $\tw$ and $\tw^\prime$ coincide.  We have
\[
{\tw}^\prime(m,\xi)=\sum_{n=1}^\infty\frac{1}{n!}{\tw}^\prime_n((m,\xi)^{\otimes n})=\left(\sum_{n=1}^\infty\frac{1}{(n-1)!}{\tw}^\prime_n(m,\underbrace{\xi,\ldots,\xi}_{n-1}),-\xi\right).
\]
On the other hand \[\tw(m,\xi)=(m^\xi,-\xi)=(m^\xi_1+m^\xi_2+\ldots,-\xi)\in\Der(\hat{S}\Pi V^*)\oplus V.\] Similarly
\[{\tw}^\prime(m,\xi)=({\tw}^\prime(m,\xi)_1+{\tw}^\prime(m,\xi)_2+\ldots,-\xi)\in \Der(\hat{S}\Pi V^*)\oplus V,\]
where ${\tw}^\prime(m,\xi)_i\in\Hom(S^i\Pi V,\Pi V)\subset\Der(\hat{S}\Pi V^*)$. We have the following formula for the structure maps of twisted $L_\infty$ structures taking into account (\ref{twist1}):
\begin{align*}
{\tw}^\prime(m,\xi)_i(x_1,\ldots, x_i)&=\sum_{n=1}^\infty\frac{1}{(n-1)!}m_{i+n-1}(\underbrace{\xi,\ldots,\xi}_{n-1},x_1,\ldots, x_i)\\
&=\sum_{n=0}^\infty\frac{1}{n!}m_{i+n}(\underbrace{\xi,\ldots,\xi}_{n},x_1,\ldots, x_i)\\
&=m^\xi_i(x_1,\ldots, x_i),
\end{align*}
where $x_1,\ldots, x_i\in\Pi V$. This shows that $\tw^\prime$ and $\tw$ agree and finishes the proof of part (1).

The proof of part (3) is similar. We denote temporarily by $\tw^\prime$ an automorphism of $\hat{S}\Pi [\tilde{\a}(V)]^*$ given by formula (\ref{twa}). It is sufficient to prove that for any formal cdga $A$ the automorphisms $\tw^\prime$ and $\tw$ induce the same maps on $A$-linear MC elements $(m,\xi)\in\MC(\tilde{\a}(V),A)$. It would follow then that $\tw$ and $\tw^\prime$ coincide.  We have
\[
{\tw}^\prime(m,\xi)=\sum_{n=1}^\infty\frac{1}{n!}{\tw}^\prime_n((m,\xi)^{\otimes n})=\left(\sum_{n=1}^\infty\frac{1}{(n-1)!}{\tw}^\prime_n(m,\underbrace{\xi,\ldots,\xi}_{n-1}),-\xi\right).
\]
On the other hand $\tw(m,\xi)=(m^\xi,-\xi)=(m^\xi_1+m^\xi_2+\ldots,-\xi)$ and we have the following formula for the structure maps of twisted $L_\infty$ structures taking into account (\ref{twist2}):
\begin{align*}
{\tw}^\prime(m,\xi)_i(x_1,\ldots, x_i)&=\sum_{n=1}^\infty\sum_{\sigma\in S_{n-1}}\frac{1}{(n-1)!}m_{i+n-1}(\underbrace{\sigma(\xi\otimes\ldots\otimes\xi)}_{n-1}|x_1,\ldots, x_i)\\
&=\sum_{n=0}^\infty m_{i+n}(\underbrace{\xi,\ldots,\xi}_{n}|x_1,\ldots, x_i)\\
&=m^\xi_i(x_1,\ldots, x_i),
\end{align*}
where $x_1,\ldots, x_i\in\Pi V$. This shows that $\tw^\prime$ and $\tw$ agree and finishes the proof of part (3).

\end{proof}
\subsection{Twisting automorphisms in operads $\L$ and $\A$}
We will now describe how MC twisting in infinity algebras gives rise to automorphisms in the operads $\L$ and $\A$.
\begin{theorem}\
\begin{enumerate}
\item There exists an automorphism $\tw$ of the dg operad $\L$ given by the formulas
\begin{align*}
&\tw(x)=-x;\\
&\tw(m_n)=\sum_{i=0}^\infty \frac{1}{i!}m_{n+i}\circ(\underbrace{x\otimes\ldots\otimes x}_i)
\end{align*}
such that for any formal cdga $A$, an $A$-linear $L_\infty$ algebra $(V,m)$ and $\xi\in\MC(V,A)$, the $A$-linear $\L$-algebra given by
the composition
of $\tw:\L\to\L$ with the operadic action map $\L\to A_+\otimes\End(\Pi V)$ determines the pair $\tw(m,\xi)=(m^\xi,-\xi)$.
\item
There exists an automorphism $\tw$ of the dg modular operad $\overline{\mathscr L}_\infty$ given by the same formula as in part (1)
and such that for any formal cdga $A$, a cyclic $A$-linear $L_\infty$ algebra $(V,m)$ and $\xi\in\MC(V,A)$, the $A$-linear $\overline{\mathscr L}_\infty$-algebra given by
the composition
of $\tw:\overline{\mathscr L}_\infty\to \overline{\mathscr L}_\infty$ with the action map of modular operads  $\overline{\mathscr L}_\infty\to A_+\otimes\End(\Pi V)$ determines the pair $\tw(m,\xi)=(m^\xi,-\xi)$.
\item There exists an automorphism $\tw$ of the dg operad $\A$ given by the formula
\begin{align*}
&\tw(x)=-x;\\
&\tw(m_n)=\sum_{i=0}^\infty\sum_{\sigma\in \operatorname{Sh(i,n)}} \sigma(m_{i+n})\circ(\underbrace{x\otimes\ldots\otimes x}_i)
\end{align*}
such that for any formal cdga $A$, an $A$-linear $A_\infty$ algebra $(V,m)$ and $\xi\in\MC(V,A)$, the $A$-linear $\A$-algebra given by
the composition
of $\tw:\A\to\A$ with the operadic action map $\A\to A_+\otimes \End(\Pi V)$ determines the pair $\tw(m,\xi)=(m^\xi,-\xi)$. Here the notation $\operatorname{Sh(i,n)}$ stands for the set of $(i,n)$-shuffles.
\item
There exists an automorphism $\tw$ of the dg modular operad $\overline{\mathscr A}_\infty$ given by the same formula as in part (3)
and such that for any formal cdga $A$, a cyclic $A$-linear $A_\infty$ algebra $(V,m)$ and $\xi\in\MC(V,A)$, the $A$-linear $\overline{\mathscr A}_\infty$-algebra given by
the composition
of $\tw:\overline{\mathscr A}_\infty\to\overline{\mathscr A}_\infty$ with the action map of modular operads  $\overline{\mathscr A}_\infty\to A_+\otimes\End(\Pi V)$ determines the pair $\tw(m,\xi)=(m^\xi,-\xi)$.

\end{enumerate}
\end{theorem}
\begin{proof}
The main issue is proving the compatibility of the automorphism $\tw$ with differentials; we start with part (3). Suppose that a discrete dg vector space $V$ has the structure of an $\A$ algebra given by a (continuous) map of operads
$f:\A\to\End(\Pi V)$ and consider the composite map
\[
\xymatrix{\A\ar^{\tw}[r]&\A\ar^-{f}[r]&\End(\Pi V).}
\]
The images $f\circ\tw(m_n), f\circ\tw(x)$ in $\End(\Pi V)$ endow $V$ with a structure of a $A_\infty$ algebra twisted by an MC element
(in the sense of Section~\ref{S:restricted})
and it follows
that $f\circ \tw$ is a dg map although $\tw$ is not yet known to be compatible with the differential. To prove the desired compatibility it suffices to find such a $V$ on which $\A$ acts faithfully. This is of course impossible, since $\A$  itself is not discrete.

To circumvent this minor difficulty, we note that $\A= \underleftarrow{\lim}_m\, \A/(x^m)$, an inverse limit of dg operads, and that
$\tw$ descends to an automorphism of $\A/(x^m)$. Thus is suffices to prove that $\tw$, viewed as an automorphism of
$\A/(x^m)$, commutes with the differential
in $\A/(x^m)$. Since $\A/(x^m)$ is discrete, its regular representation
$V=\oplus_{n=0}^\infty(\A/(x^m))(n)$ is also discrete (and of course faithful), and then the argument suggested above applies, giving the
desired conclusion.

Let us now show that $\tw$ preserves the cyclic structure. In the $\L$ case this boils down to proving that the element $\sum_{i=0}^\infty \frac{1}{i!}m_{n+i}\circ(\underbrace{x\otimes\ldots\otimes x}_i)$ is cyclically invariant but this is clear since by the cyclicity of $L$ each individual term $m_{n+i}\circ(\underbrace{x\otimes\ldots\otimes x}_i)$ is cyclically invariant. Therefore $\tw$ is an automorphism of a cyclic operad $\L$ and thus determines an automorphism of its modular closure $\overline{\mathscr L}_\infty$.

Similarly, in the $\A$ case we are required to prove that $\sum_{\sigma\in \operatorname{Sh(i,n)}} \sigma(m_{i+n})\circ(\underbrace{x\otimes\ldots\otimes x}_i)$ is invariant with respect to the action of the $n$-cycle and this follows from the fact
that the latter determines a bijection of the set of $(i,n)$-shuffles onto a conjugate set of shuffles.

This completes the proof that $\tw$ is a dg automorphism in all four cases. The claim that the $A$-linear $\L$ algebra or $\A$ algebra structure
on $A\otimes V$ obtained by pulling back along $\tw$ is the MC-twisted one is a direct consequence of formulas (\ref{twist1}) and (\ref{twist2}). The cyclic analogue is likewise clear.
\end{proof}

\section{Twisting in general operads}\label{secw}
Until now we studied the concept of MC twisting in infinity-algebras and operads governing twisted structures. Since an operad itself is a generalization of an algebra it is natural to ask whether it is possible to twist operads themselves. It is not clear that a full-fledged generalization is possible; here we only consider a very restricted notion of twisting by MC elements corresponding to unary operations.
Our main example of such a twisting is a functor $\O\mapsto\hat{\O}$ which allows one to consider the underlying differential
on a given dg vector space as part of an operadic structure.
\subsection{Operad twisting}
Let $\O$ be a dg operad. Associated to $\O$ is the dgla $\Der(\O)$ of derivations of $\O$ consisting of collections of linear maps $\xi_n:\O(n)\to\O(n), \, n=0,1,\ldots$ such that $\xi_n(f\circ_i g)=(\xi_nf)\circ_i g+(-1)^{|f||\xi_n|}f\circ_i \xi_n(g)$ where $f\in\O(k),\, g\in\O(l),\, i=1,\ldots l$. It is clear that the commutator of derivations is again a derivation and the differential in $\O$ induces one on $\Der(\O)$; thus the latter is indeed a dgla. An MC element $\xi$ in it allows one to twist the differential $d_\O$ in $\O$ by the formula $d^\xi:=d_\O+\xi$.

Consider a special case of this construction corresponding to an \emph{inner} derivation of $\O$. Note that the dga $\O(1)$ acts on $\O$ as operad derivations according to the formula
\[
\delta_a(z)=a\circ_1 z-(-1)^{|a||z|}\sum_{i=1}^nz\circ_ia
\]
where $z\in\O(n)$ and $a\in\O(1)$. Then the correspondence $a\mapsto\delta_a$ determines a a dgla map $\O(1)\to\Der(\O)$. If $m$ is an MC element in $\O(1)$, i.e. $m$ satisfies the equation $dm+\frac{1}{2}[m,m]=0$, then we obtain the twisted dg operad $\O^m$ having the same underlying operad as $\O$ and whose differential is $d^{\delta_m}$.
\begin{example}
Let $(V,d_V)$ be a dg vector space. Then an MC element in the dga $\End(V,d_V)(1)$ is an odd endomorphism $m:V\to V$ such that $d_Vm+md_V+m^2=0$, i.e. such that $(d_V+m)^2=0$. The correspondence $m\mapsto d_V+m$ determines a bijection between the set of MC elements in $\End(V,d_V)(1)$ and the set of differentials on the graded vector space $V$. Furthermore,
\[
[\End(V,d_V)]^m=\End(V,d_V+m).
\]
\end{example}
Now let $\O\to\End(\Pi V)$ be an $\O$ algebra structure on a dg space $(V,d_V)$ and $m\in\O(1)$ be an MC element which is also viewed as an odd operator $\Pi V\to \Pi V$. Then we have an induced map
\begin{equation}\label{twop}
\O^m\to[\End(\Pi V)]^m=\End(\Pi V,d_V+m)
\end{equation} Conversely, such a map (\ref{twop}) of twisted operads can be `untwisted' by the MC element $-m\in\O^m(1)$ giving an $\O$-algebra structure on $(V,d_V)$.
We obtain the following result.
\begin{prop}\label{twchar}
Let $\O$ be a dg operad and $m$ be an MC element in $\O(1)$. Then an $\O$-algebra structure on a dg vector space $(V,d_V)$ is equivalent to an $\O^m$ algebra structure on $(V,d_V+m)$.
\end{prop}
\noproof

\subsection{The hat construction}
\begin{defi}
Let $\O$ be a formal dg operad and consider the formal operad $\O[m]$ generated freely by $\O$ and $m\in\O[m](1)$. The differential in $\O[m]$ is fixed by requiring that $\O$ is a dg suboperad in $\O[m]$ and $d(m)=-m^2$. Note this makes $m$ an MC element of $\O[m]$.

Then define the hat-construction $\hat{\O}$ as $\hat{\O}:=(\O[m])^{m}$, the twisting of $\O[m]$ by the MC element $m$.
\end{defi}
\begin{rem}\label{cobarhat}
Let $\O=\mathsf{B}\P$ be the cobar-construction operad of a dg operad $\P$. Then it follows directly from the definition that
\begin{itemize}
\item $\O[m]\cong \mathsf{B}(\P\oplus\ground \langle e\rangle)$ where $\P\oplus\ground\langle e\rangle$ is the operad $\P$ with one extra generator  $e$ in arity one having zero compositions with all elements of $\P$ and such that $x\circ_1x=x$;
\item $\hat{\O}\cong \mathsf{B}(\P_e)$ where $\P_e$ is the operad $\P$ with an adjoined unit (i.e. the underlying space of $\P_e$ is the same as that of $\P\oplus\ground \langle e\rangle$ but this time operadic compositions make $e$ a unit). In particular:
\begin{enumerate}
\item $\hat{l}_\infty\cong L_\infty$.
\item $\hat{a}_\infty\cong A_\infty$.
\end{enumerate}
\end{itemize}
\end{rem}
\begin{prop}\label{hatchar}
Let $\O$ be a dg operad and $(V,d)$ be a dg vector space. Then the set of $\hat{\O}$ algebra structures on $(V,d)$ is in one-to-one correspondence
with the set of $\O$ algebra structures on $(V,d^\prime)$ where $d^\prime$ is any differential on $V$.
\end{prop}
\begin{proof}
Let $f:\hat{\O}\to\End(V,d)$ be a dg operad map giving $(V,d)$ the structure of an $\hat{\O}$ algebra. Then $g:=f|_\O$ determines a map of dg operads
$\O\to\End(V,d-f(m))$. Conversely, any dg operad map $g:\O\to\End(V,d^\prime)$ can be extended to $f:\hat{\O}\to\End(V,d)$ by setting $f(m)=d-d^\prime$.
\end{proof}
\begin{rem} We conclude that the set of $\hat{\O}$ algebra structures on $(V,d)$ does not depend on $d$ and thus, $d$ can be taken to be equal to zero.
In other words, any $\hat{\O}$ structure on $(V,d)$ determines, and is determined by, an $\hat{\O}$ algebra structure on $(V,0)$. It is, thus, sufficient to consider only $\hat{\O}$ algebra structures on graded vector spaces (i.e. with vanishing differential). Equivalently, one can consider algebra structures on $V$ over the smaller operad $\O$, but this time varying a differential on $V$ leads to genuinely different structures.  In particular, $l_\infty$ and $a_\infty$ structures on dg vector spaces are equivalent, respectively, to $L_\infty$ and $A_\infty$ structures on graded vector spaces.
\end{rem}

\begin{theorem}\label{hatacyclic} Let $\O$ be a dg operad.
\begin{enumerate}
\item
The inclusion of operads $\O \hookrightarrow \O[m]$ is a quasi-isomorphism.
\item The operad $\hat{\O}$ is acyclic.
\end{enumerate}
\end{theorem}
\begin{proof}
The operad $\O[m]$ may be understood as the free product of $\O$ with the acyclic nonunital dga $m\ground[m]\subset \ground[m]$
 (with differential $d(m)=-m^2$).
 Since $\ground$ is of characteristic zero, it follows that the operad inclusion $\O \hookrightarrow \O[m]$ is a quasi-isomorphism.
 For part (2), first note that
$\hat{\O}$ is a direct sum of $m\ground[m]$ (with the twisted differential $d^m(m)=m^2$)
and the two-sided operadic ideal $\Q$ generated by $\O$.
Let $\P$ be the operadic right ideal generated by $\O$; then
there is a decomposition
\begin{align*}
\Q & = \P \oplus (m\ground[m]\circ\P)  \\
&\cong \P \oplus (m\ground[m]\otimes\P)  \\
   &\cong \left(\ground\oplus m\ground[m]\right)\otimes \P \\
   &\cong \ground[m]\otimes \P. \\
\end{align*}
Here, by $\ground[m]\otimes\P$ we understand a collection (not an operad) such that
$(\ground[m]\otimes\P)(n)=\ground[m]\otimes\P(n)$, and similarly for the other tensor products appearing in the above isomorphisms.
Under this identification, the differential on $\Q$ is a $(\ground[m],d^m)$-linear derivation
such that
  for any $p\in\P$ we have
$d_{\Q}(1 \otimes p)=1\otimes p'+m\otimes p$ for some $p'\in\P$.
Hence
\begin{align*}
d_{\Q}(m^r\otimes p) &= (-1)^r m^r d_{\Q}(1\otimes p) +  d^m(m^r)\otimes p \\
&=  (-1)^r m^r \left( 1\otimes p'+m\otimes p\right) + d^m(m^r)\otimes p  \\
&=  (-1)^r  \left( m^r\otimes p'+m^{r+1}\otimes p\right) +
\begin{cases}
0 & \text{if $r$ is even} \\
m^{r+1} \otimes p & \text{if $r$ is odd}
\end{cases} \\
&= (-1)^r  m^r\otimes p'
+
\begin{cases}
m^{r+1}\otimes p & \text{if $r$ is even} \\
0 & \text{if $r$ is odd}
\end{cases}
\end{align*}
We see, therefore, that $d_{\Q}= d_1\otimes \id + \id \otimes d_2$, where $d_1$ is an acyclic differential on $\ground[m]$,
and $d_2$ is a differential on $\P$ whose precise form is unimportant. It follows that $\Q$, and therefore $\hat{\O}$, is acyclic, as required.
\end{proof}
\begin{rem}\
\begin{enumerate}\item
If $\O$ is a cobar-construction operad then $\hat{\O}$ is a cobar-construction of a \emph{unital} operad which is, thus, acyclic. The argument showing acyclicity of the hat-construction in general (part (2) of Theorem \ref{hatacyclic}) turns out to be considerably more involved.
\item
If $\O$ is a dg modular operad then one constructs a dg modular operad $\hat{\O}:=\O[m]^{m}$ where $m\in\O((0,2))$ and $m$ is $S_2$-invariant. All results of this section concerning $\hat{\O}$ have obvious analogues in the modular context, with one important exception: the vacuum part $\hat{\O}((0))$ of the hat-construction of $\O$ is \emph{not} acyclic but quasi-isomorphic to the direct sum of $\O((0))$ and the complex consisting of bivalent graphs (polygons) with the usual expanding differential. The reason for this is that the twisting by $m$ in $\O[m]((0))$ amounts to doing nothing as the unary operation $m$ acts trivially on the vacuum part; thus $\hat{\O}((0))\cong \O[m]((0))$. In particular, this explains a well-known fact that the graph complex (commutative, ribbon or Lie) containing bivalent vertices is a direct sum of the usual graph complex (with vertices of valence three or higher) and the polygonal complex.
\end{enumerate}
\end{rem}

\section{MC elements in operadic algebras}

In this section we assume that $\O$ is a dg operad supplied with a dg operad map $\phi:L_\infty\to\O$. In this situation there is a notion of an MC element in an $\O$ algebra.
\begin{defi} Let $A$ be a formal cdga and $V$ be an $A$-linear $\O$ algebra.
An MC element (more precisely, a $\phi$-MC element) in $V$ with coefficients in $A$ is an MC element in $A_+\otimes V$ viewed as an $L_\infty$ algebra by pulling back along the map $\phi:L_\infty\to\O$. The set of such MC elements will be denoted by $\MC_{\phi}(V,A)$ or simply by $\MC(V,A)$ if the map $\phi:L_\infty\to\O$ is clear from the context.
\end{defi}
\begin{rem}
Recall that there is a canonical map $L_\infty\to A_\infty$; it is then straightforward to see that the above definition of an MC element in an $A_\infty$ algebra is consistent with one considered earlier in the present paper.
\end{rem}
We now define the operad $\O_{\MC}$ governing $\O$ algebras together with a choice of an MC element.
\begin{defi}
The dg operad $\O_{\MC}$ has $\O[[x]]$ as the underlying graded operad, where $x\in\O[[x]](0)$. The differential of $\O_{\MC}$ is specified
by requiring that $\O\subset\O_{\MC}$ be a dg suboperad and by the formula
\[
d_{\O_{\MC}}(x)=\sum_{n=1}^\infty\frac{1}{n!}\phi(m_n)\circ x^{\otimes n}.
\]
\end{defi}
\begin{rem}
Let $\O=L_\infty$ and $\phi$ be the identity morphism. Then $(L_\infty)_{\MC}$ is by definition the operad $\L$ introduced in Definition
\ref{defLin}. This observation also implies that $d_{\O_{\MC}}$ squares to zero. Indeed, by construction we have a unique morphism $\phi^\prime:\L\to\O_{\MC}$
extending $\phi$ and sending $x\in\L$ to $x\in\O_{\MC}$. Then $ d_{\O_{\MC}}\circ\phi^\prime
=\phi^\prime\circ d_{\L} $ and it follows
that $d_{\O_{\MC}}^2(x)=0$ from which the desired conclusion follows.
\end{rem}
We now have the following straightforward generalization of Theorem \ref{twalg}.
\begin{theorem}\label{charMC}
Let $V$ be an $A$-linear $\O$ algebra and $\xi\in \MC(V,A)$. Then $V$ has the structure of an $A$-linear $\O_{\MC}$ algebra. Conversely, an $A$-linear $\O_{\MC}$ algebra gives rise to an $A$-linear $\O$ algebra $V$ together with a choice of an element in $\MC(V,A)$.
\end{theorem}
\noproof

We now consider the special case of the above construction in which the given operad map $\phi:L_\infty\to\O$ factors through
$l_\infty$; abusing notation, the ensuing map $l_\infty\to\O$ will also be denoted by $\phi$.
We will see that in this situation the functor $\O_{\MC}$ has better homotopical properties. This special case is
also the basis of our interpretation of Willwacher's construction $\Tw\O$ given in the next section.
\begin{prop}\label{linftyMChat} \
Let $\phi:l_\infty\to \O$ and $\psi:l_\infty\to \P$ be maps of dg operads, and let $\rho:\O\to\P$ be a map of dg operads
such that $\rho\circ\phi=\psi$. If $\psi$ is a quasi-isomorphism, then the induced map
$\psi_{\MC}:\O_{\MC}\to\P_{\MC}$ is a quasi-isomorphism.
\end{prop}
\begin{proof}
The dg operad $\O_{\MC}$ is complete with respect to the $x$-adic filtration, and the associated graded dg operad is
isomorphic to
$\O[[x]]$, with differential equal to $d_{\O}$ on $\O$ and vanishing on $x$. The map $\psi_{\MC}:\O_{\MC}\to\P_{\MC}$ induces
the map $\psi[[x]]: \O[[x]]\to\P[[x]]$ on associated graded dg operads. Under our assumption that the characteristic
of $\ground$ is zero, the fact that $\psi$ is a quasi-isomorphism implies that $\psi[[x]]$ is a quasi-isomorphism. It follows that $\psi_{\MC}$ is also a quasi-isomorphism.
\end{proof}
\begin{rem}
Proposition \ref{linftyMChat} could be be phrased by saying that the functor $\O\mapsto\O_{\MC}$ is quasi-isomorphism invariant.
Note however, that this holds only for operads $\O$ for which the structure map $L_\infty\to\O$ factors through $l_\infty$. Consider,
by contrast, the operad $L_\infty$ itself. We shall see that $(L_\infty)_{\MC}$ is acyclic. On the other hand $L_\infty$ itself is
acyclic, i.e. it is quasi-isomorphic to the zero operad $0$ while $0_{\MC}$ is clearly not acyclic -- it is one-dimensional in arity zero.
\end{rem}
Let us consider the case $\O=l_\infty$. Note that $l_\infty$ is quasi-isomorphic to $\Lie$, the (non-unital) operad governing Lie algebras. It follows that $(l_\infty)_{\MC}$ is quasi-isomorphic to $\Lie_{\MC}$. Similarly $(a_\infty)_{\MC}$ is quasi-isomorphic to $\Ass_{\MC}$ where $\Ass$ is the operad of associative algebras. However, more is true.
\begin{theorem}\label{MCquasi}\
\begin{enumerate}
\item The natural split inclusions $a_\infty\hookrightarrow(a_\infty)_{\MC}$  and $l_\infty\hookrightarrow (l_\infty)_{\MC}$ are quasi-isomorphisms of dg operads.
\item The natural split inclusions $\Ass\hookrightarrow\Ass_{\MC}$  and $\Lie\hookrightarrow\Lie_{\MC}$ are quasi-isomorphisms of dg operads.
\end{enumerate}

\end{theorem}
\begin{proof}
First note that part (1) is a direct consequence of part (2) by virtue of Proposition~\ref{linftyMChat}.
Let us now prove part (2) in the $\Lie$ case. The dg vector vector space $\Lie_{\MC}(0)$ is spanned
by the elements $x$ and $m_2(x,x)$ with the differential acting as $x\mapsto m_2(x,x)$; it is, thus, acyclic.

Now consider the case $n>0$. Recall that $\Lie(1)=0$ while for $n\geq 2$ the space $\Lie(n)$ can be identified with the space spanned inside the free algebra on $n$ generators $a_1,\ldots,a_n$ by the Lie monomials in these generators containing each of them exactly once.  Using the Lyndon basis inside a free Lie algebra (see e.g. \cite{Reu} concerning this notion) we see that $\Lie(n)$ has a basis consisting of operations of the form $[a_{\sigma(1)}[a_{\sigma(2)}\ldots[a_{\sigma(n-1)},a_n]\ldots]]$. Here $\sigma\in S_{n-1}$. To alleviate the notation we will write $a\cdot b$ for $[a,b]$ and omit the brackets in monomials. Our convention will be that $a\cdot b\cdot\ldots \cdot c\cdot d:=a\cdot(b\cdot(\ldots (c\cdot d)\ldots ))$.

Thus, we write our basis in $\Lie(n), n\geq 2$ as $\{a_{\sigma(1)}\cdot \ldots \cdot a_{\sigma(n-1)}\cdot a_n\}_{\sigma\in S_{n-1}}$; moreover the subgroup $S_{n-1}\subset S_n$ acts on such expressions by permuting the first $n-1$ symbols. This does not specify the full $S_n$-action on $\Lie(n)$ but is sufficient for our purposes. We have
\[
\Lie[[x]](n)=\Lie(n)\oplus\prod_{k=1}^\infty\Lie(k+n)_{S_k}.
\]
Recall that a basis in $\Lie(n+k)$ consists of monomials in $a_1,\ldots, a_{n+k}$ where $a_{n+k}$ is placed last and the remaining $a_i$s can be arbitrarily permuted. This gives rise to a basis in $\Lie(k+n)_{S_k}$ for $k\geq 1$. Namely, it consists of monomials in $a_i$s as above, with each of the $k$ symbols $a_1, \ldots, a_k$
replaced by the symbol $x$. Redenoting the symbol $a_{n+i}$ as $a_i$ for $1\leq i \leq k$  we can write a typical element in the constructed basis as $x^{\cdot k_1}\cdot a_{\sigma(1)}\cdot x^{\cdot k_2}\cdot a_{\sigma(2)}\cdot\ldots\cdot x^{\cdot k_l}\cdot a_{\sigma(n-1)}\cdot a_n$ where $k_i\geq 0, \sum_{i=1}^lk_i=k$ and $\sigma\in S_{n-1}$.

In other words all our monomials are obtained from $x^{\cdot k}\cdot a_1\cdot\ldots\cdot a_{n-1}\cdot a_n$ by permuting the symbols $x$ and $a_i,i=1,\ldots, n-1$ in all possible ways. This specifies a (topological) basis in $\Lie[[x]](n)$.

Further, the differential in $\Lie_{\MC}(n)\cong \Lie[[x]](n)$ is given formally by $d(x)=x\cdot x$ and $d(a_i)=0,i=1,\ldots,n$ and the Leibniz rule with respect to the product $\cdot$. The differential vanishes on the direct summand $\Lie(n)\subset\Lie_{\MC}(n)$ as it should. It suffices, therefore, to prove that the dg space $\prod_{k=1}^\infty\Lie(k+n)_{S_k}$ is acyclic. We define a contracting homotopy $s$ on the basis consisting of monomials in $x$ and $a_i,i=1,\ldots,n$ described above as follows. Given a monomial $M=x^{\cdot k_1}\cdot a_{\sigma(1)}\cdot x^{\cdot k_2}\cdot a_{\sigma(2)}\cdot\ldots\cdot x^{\cdot k_l}\cdot a_{\sigma(n-1)}\cdot a_n$ choose $s$ such that $k_s$ is the first positive integer among $k_1,\ldots, k_l$ and set
\[
s(M)=x^{\cdot k_1}\cdot a_{\sigma(1)}\cdot x^{\cdot k_2}\cdot a_{\sigma(2)}\cdot\ldots\cdot x^{\cdot k_s-1}\cdot\ldots\cdot x^{\cdot k_l}\cdot a_{\sigma(n-1)}\cdot a_n
\]
if $k_s$ is even and $s(M)=0$ if $k_s$ is odd.
In other words the cluster $\underbrace{x\cdot\ldots\cdot x}_{k_s}$
in the monomial $M$ gets replaced by the cluster $\underbrace{x\cdot\ldots\cdot x}_{k_s-1}$, if $k_s$ is even. This concludes the proof in the $\Lie$ case.

The proof in the $\Ass$ case is similar but simpler. The case $n=0$ does not have to be considered separately. The space $\Ass(0)$ is zero while for $n>1$  the space $\Ass(n)$ is spanned by tensor monomials in $n$ generators $a_1,\ldots,a_n$ which contain every such generator exactly once; they are therefore of the form $a_{\sigma(1)}\cdot\ldots\cdot a_{\sigma(n)}$ where $\sigma\in S_n$. Further, we have an isomorphism $\Ass[[x]](n)\cong \prod_{k=0}^\infty\Ass(k+n)_{S_k}$ and the space $\Ass(k+n)_{S_k}$ has a basis of monomials of length $n+k$ containing the symbols $a_1,\ldots, a_n$ and $k$ copies of $x$. The differential in $\Ass_{\MC}$ is specified on such monomials by the same formula as in the $\Lie$ case and the same contracting homotopy can be used to show that the direct complement of $\Ass$ inside $\Ass_{\MC}$ is acyclic.

\end{proof}
\begin{rem}
The statement of Theorem \ref{MCquasi} can be restated as the existence of commutative diagrams of operads where all arrows are quasi-isomorphisms.
\[
\xymatrix
{
a_\infty\ar[r]\ar[d]&(a_{\infty})_{\MC}\ar[d]\\
\Ass\ar[r]&\Ass_{\MC}
}\hspace{2cm}
\xymatrix
{
l_\infty\ar[r]\ar[d]&(l_{\infty})_{\MC}\ar[d]\\
\Lie\ar[r]&\Lie_{\MC}
}
\]
\end{rem}
We will now see that the dg operads $(A_\infty)_{\MC}$ and $(L_\infty)_{\MC}$ are acyclic.
\begin{lem}
Let $\phi:l_\infty\to \O$ be a map of dg operads.
Then we have an isomorphism of dg operads $\widehat{\O_{\MC}}\cong \hat{\O}_{\MC}$, where the MC construction on the
right-hand side is taken with respect to the induced map $\hat{\phi}: L_\infty=\hat{l}_\infty \to \hat{\O}$.

\end{lem}
\begin{proof}
Both $\widehat{\O_{\MC}}$ and $\hat{\O}_{\MC}$ are isomorphic, as graded operads, to $\O[m][[x]]$, the dg operad obtained
from $\O$ by adjoining a $0$-ary operator $x$ and a $1$-ary operator $m$ and completing at the ideal generated by $x$. A straightforward
calculation shows that their differentials agree on $\O$, $x$ and $m$, and hence they are isomorphic as dg operads.
\end{proof}
\begin{cor}
The operads $\L$ and $\A$ are acyclic.
\end{cor}
\begin{proof}
We have the isomorphism $\L\cong \left(L_{\infty}\right)_{\MC} \cong \widehat{\left(l_\infty\right)_{\MC}}$. Since by Theorem \ref{hatacyclic} the hat-construction of any operad is acyclic we conclude that $\L$ is acyclic. The proof for $\A$ is similar.
\end{proof}

\section{Twisting in operadic algebras}\label{operal}
Given an $A$-linear $\O$ algebra structure on $V$ together with an MC element $\xi\in\MC(V,A)$ we can consider the twist of the underlying
$A$-linear $L_\infty$ structure by $\xi$. If there is an endomorphism of $\O_{\MC}$ extending the automorphism $\tw$ of $\L$ then the twisted higher $L_\infty$ products on $A\otimes V$ are themselves part of an $\O$ algebra structure.  In general, such an endomorphism does not exist. For example, there is no endomorphism of $(l_\infty)_{\MC}$ extending the endomorphism $\tw$ on $(L_\infty)_{\MC}=\L$. However, we can still twist the differential on $A\otimes V$ and ask whether this new dg vector space supports
an operadic structure. This answer to this is provided by the notion of Willwacher's twisting of the operad $\O$ introduced  in \cite{Wil}, cf. also a detailed exposition in \cite{Dol}; we give
an alternative, hopefully more transparent, treatment.

\begin{lem}
Let $\alpha:=-\sum_{n=1}^\infty\frac{1}{n!}m_{n+1}\circ x^{\otimes n}\in [l_\infty]_{\MC}(1)$. Then $\alpha$ is an MC element in the dga $[l_\infty]_{\MC}(1)$.
\end{lem}
\begin{proof}
Note that $\alpha=m_1-\tw(m_1)$. Since $-m_1$ is an MC element in the dga $[\hat{l}_\infty]_{\MC}(1)=\L(1)$ and $\tw$ is an automorphism of $\L(1)$ we conclude that $\tw(-m_1)=-\tw(m_1)$ is also an MC element. Upon twisting by $-m_1$ in the dga $\L(1)$ the MC element $-\tw(m_1)$ becomes an MC element $m_1-\tw(m_1)=\alpha$ inside $\L^{-m_1}$ and it remains to observe that $\alpha$ belongs to the dg subalgebra $[l_\infty]_{\MC}(1)\subset\L^{-m_1}$.
\end{proof}
The following result is an immediate consequence.
\begin{cor}
Let $\alpha:=-\sum_{n=1}^\infty\frac{1}{n!}\phi(m_{n+1})\circ x^{\otimes n}\in \O_{\MC}(1)$. Then $\alpha$ is an MC element in the dga $\O_{\MC}(1)$.
\end{cor}
\noproof
We can now give a definition of an operad `governing' $\O$ algebras with an MC twisted differential.
\begin{defi}
Let $\O$ be a dg operad supplied with a dg operad map $\phi:l_\infty\to\O$. The dg operad $\Tw_{\phi}\O$ is defined as $\Tw_{\phi}\O=[\O_{\MC}]^\alpha$. When the map $\phi$ is clear from the context it will be suppressed from notation.
\end{defi}
\begin{rem}
Our version of $\Tw\O$ is almost the same as Willwacher's in \cite{Wil}, which we denote here by $\Tw^W\O$. Namely, $\Tw^W\O$ is a dg suboperad in $\Tw\O$ and under this inclusion $\Tw^W\O(n)\cong \Tw\O(n)$ for $n>0$. In arity zero we have
\[{\Tw}^W\O(0)=\O(0)\hookrightarrow\O(0)\oplus\ground\langle x\rangle={\Tw}\O(0)\]
where $\ground\langle x\rangle$ stands for the one-dimensional space spanned by $x$.

Willwacher's construction starts with the operad $\widetilde{\Tw}\O$ which is our $\O[[x]]$ in arities $>0$ and $\O(0)$ in arity zero. In order to arrive at $\O_{\MC}$ the differential in $\widetilde{\Tw}\O$ is twisted by the MC element $\phi$ in the convolution Lie algebra of maps from $l_\infty$ to $\O$ which is shown to act on the Lie algebra of derivations of $\widetilde{\Tw}\O$. Finally, the twisting by $\alpha$ is performed to obtain $\Tw^W$.
\end{rem}
The following result is analogous to Proposition \ref{linftyMChat} and is proved using the same argument.
\begin{prop}\label{linftyMChat'} \
Let $\phi:l_\infty\to \O$ and $\psi:l_\infty\to \P$ be maps of dg operads, and let $\rho:\O\to\P$ be a map of dg operads
such that $\rho\circ\phi=\psi$. If $\psi$ is a quasi-isomorphism, then the induced map
$\psi_{\MC}:\Tw_\phi\O\to \Tw_\phi\P$ is a quasi-isomorphism.
\end{prop}
\noproof
In particular, we see that operads of the form $\Tw\hat{\O}$ (such as $\Tw L_\infty$) are acyclic. The following result computes the homology of $\Tw l_\infty$ and $\Tw a_\infty$; it is analogous to the statement of Theorem \ref{MCquasi} and relies on it.
\begin{theorem}\label{MCquasi'}\
\begin{enumerate}\item
The natural projections $\Tw a_\infty\to a_\infty$ and $\Tw l_\infty\to l_\infty$ sending $x\in\Tw l_\infty(0)$ or $x\in\Tw a_\infty(0)$ to zero are quasi-isomorphisms of dg operads.
\item
The natural projections $\Tw\Ass\to \Ass$ and $\Tw\Lie\to \Lie$ sending $x\in\Tw\Ass(0)$ or $x\in\Tw\Ass(0)$ to zero are quasi-isomorphisms of dg operads.
\end{enumerate}
\end{theorem}
\begin{proof}
First note that part (1) is a direct consequence of part (2) by virtue of Proposition~\ref{linftyMChat'}.
We now prove part (2), restricting ourselves to the $\Lie$ case, the proof of the $\Ass$ case being essentially the same.

Recall that $\Tw\Lie:=[\Lie_{\MC}]^\alpha$ where $\alpha=-m_2(x,-)\in\Lie_{\MC}(1)$ and $m_2\in\Lie_{\MC}(2)$ . Further, $\Lie_{\MC}$ is isomorphic to $\Lie[[x]]$ where $x\in\Lie(0)$ and the differential acts as $d(x)=\frac{1}{2}m_2(x,x)$; here $m_2$ is the binary generator of $\Lie$ (and, of course, $d(m_2)=0$). The twisted differential will have the form \[d^\alpha(x)=\frac{1}{2}m_2(x,x)-m_2(x,x)=-\frac{1}{2}m_2(x,x).\]
Next, using the symbols $a,b$ as placeholders we find
\[d^\alpha(m_2)(a,b)=[x,[a,b]]-[[x,a],b]-[a,[x,b]].  \]
The last expression is zero by the Jacobi identity and so $d^\alpha(m_2)=0$. It follows that the map $\Tw l_\infty\cong (l_\infty)_{\MC}^\alpha\to (l_\infty)_{\MC}$ with
$m_2\mapsto -m_2$ and $x\mapsto x$ is an isomorphism of dg operads and the conclusion follows from Theorem \ref{MCquasi}.
\end{proof}
\begin{rem}
Theorem \ref{MCquasi'} could be phrased as saying that there exists a commutative diagrams of dg operads all arrows of which are quasi-isomorphisms.
\[
\xymatrix
{
\Tw l_\infty\ar[r]\ar[d]&l_\infty\ar[d]\\
\Tw\Lie\ar[r]&\Lie
}\hspace{2cm}
\xymatrix
{
\Tw a_\infty\ar[r]\ar[d]&a_\infty\ar[d]\\
\Tw\Ass\ar[r]&\Ass
}
\]
\end{rem}
\subsection{Block modules and algebras}\label{Block}
To understand properly algebras over the operad $\Tw_\phi\O$ we need a certain generalization of the notion of an $A$-linear $\O$ algebra which is of some independent interest. This generalization is obtained by treating $A$ consistently as the `ground' ring, so that it completely supplants $\ground$. We note that all our results involving $A$-linear structures could be formulated in this more general framework but for most applications it is not needed which is the reason it has not been adopted throughout the paper.
\begin{defi}
Let $(A,d_A)$ be a cdga and $(V,d_V)$ be a dg vector space. Let $d\in\operatorname{End}(A\otimes V)$ be an $A$-linear derivation, i.e. $d(a\otimes v)=d_A(a)\otimes v+(-1)^{|a|}a\otimes d(v)$; additionally assume that $d^2=0$. The pair $(A\otimes V,d)$ is called a \emph{Block module}.
It is called a \emph{formal} Block module if, in addition,
$A$ is a formal cdga and $d(v)-d_V(v)\in A_+\otimes V$ for all $v\in V$.
\end{defi}

\begin{rem}
Block modules were introduced (in somewhat greater generality) in \cite{block}. We will only be concerned with
formal Block modules. Given a formal cdga $A$ and a dg vector space $V$, the formal Block module structures
on $A\otimes V$ are in 1-1 correspondence with elements in $\MC(\operatorname{End}(V),A)$; namely, given $\xi\in\MC(\operatorname{End}(V),A)$, the endomorphism $d=d_A\otimes \id + \id \otimes d_V + \xi$ is a Block
differential on $A\otimes V$.
\end{rem}

\begin{defi}
Let  $(A\otimes V, d)$
be a formal Block module. Form the suboperad $\End_A(A\otimes V,d) \subset
\End(A\otimes V,d)$
consisting of $A$-multilinear maps $\left\{ (A\otimes  V)^{\otimes n} \to
A_+ \otimes V\right\}$; note that
$\End_A(A\otimes V,d)(n) \cong A_+ \otimes \Hom(V^{\otimes n}, V)$.

Further, given  a dg operad $\O$, a Block $\O$ algebra structure on $(A\otimes V, d)$
is an operad map $\O \to \End_A(\Pi A\otimes V,d)\cong \End_A(A\otimes\Pi V,d)$.
\end{defi}

\begin{rem}
Note that the operad  $\End_A(A\otimes V,d)$ is obtained from $A_+\otimes\End(V,d_V)$ by
twisting with the MC-element $\xi \in \MC(\End(V,d_V)(1),A)$ associated to the
Block module $(A\otimes V,d)$. In particular, the zero MC element gives a trivial twisting
of the operad $A_+\otimes\End(V,d_V)$; the associated notion of a Block $\O$ algebra then reduces
to the notion of an $A$-linear $\O$ algebra considered earlier in the paper.
\end{rem}

\begin{prop}\label{Blocktwchar}
Let $\O$ be a dg operad and $m$ be an MC element in $\O(1)$. Then the structure of a Block $\O$ algebra on a
formal Block module $(A\otimes V,d)$ is equivalent to an $\O^m$ algebra structure on the Block module
$(A\otimes V,d+m)$.
\end{prop}
\begin{proof}
Let $\xi\in A_+\otimes\End(V)(1)$ be the MC-element associated to a Block module $(A\otimes V,d)$.
Suppose we are given an operad map $\O \to \End_A(A\otimes \Pi V,d)$. Then
the image of $m\in\O(1)$ in $\End_A(A\otimes \Pi V,d)(1) = \left[A_+\otimes\End(\Pi V)(1)\right]^\xi$ is an MC-element which we denote by the same symbol. So we obtain an induced operad map
$\O^m \to  \left[A_+\otimes\End(\Pi V)\right]^{\xi+m} = \End_A(A\otimes \Pi V,d+m)$. This establishes the
desired 1-1 correspondence.
\end{proof}
\begin{defi}
Let $\O$ be a dg operad supplied with an operad map $\phi:L_\infty\to\O$ and $(A\otimes V,d)$ be a Block $\O$ algebra.
Then an even element $\xi\in A_+\otimes \Pi V$ is
called a $\phi$-MC element if $d(\xi)+\sum_{i=1}^\infty
\frac{1}{i!}\phi(m_i)(\xi^{\otimes i})=0$. The set of such MC elements will be denoted by $\MC_{\phi}(V,A)$ or simply by $\MC(V,A)$ if the map $\phi:L_\infty\to\O$ is clear from the context.
\end{defi}
We have the following straightforward analogue of Theorem \ref{charMC}.
\begin{prop}\label{charMC1}
Let $(A\otimes V,d)$ be an Block $\O$ algebra and $\xi\in \MC(V,A)$. Then $(A\otimes V,d)$ has the structure of a Block $\O_{\MC}$ algebra. Conversely, any Block $\O_{\MC}$ algebra gives rise to an Block $\O$ algebra $(A\otimes V,d)$ together with a choice of an element in $\MC(V,A)$.
\end{prop}
\noproof

Then Proposition \ref{Blocktwchar} and Proposition \ref{charMC1} yield the following characterization of Block $\O$ algebras over $\Tw\O$.
\begin{theorem}
A Block $\Tw\O$ algebra structure on  $(A\otimes V,d)$ is equivalent to a Block $\O$ algebra structure on $(A\otimes V,d^\prime)$
together with $\xi\in\MC_{\phi}(V,A)$ such that $d$ is the MC-twisted differential, i.e.
$d=d^\prime+\sum_{n=1}^\infty \frac{1}{n!}m_{n+1}(\xi,\ldots,\xi,-)$.
\end{theorem}
\noproof
\begin{rem}\
\begin{enumerate}
\item
Since an $A$-linear $\Tw\O$ algebra is a special case of a Block algebra the above result also gives a characterization of $A$-linear $\Tw\O$ algebra structures. However, even if the original Block $\Tw\O$ algebra happens to be $A$-linear, the associated Block $\O$ algebra, in general, will not be. This was our chief reason for introducing Block modules and algebras.
\item
Since the dg operad $\Tw\O$ acts on $\O$ algebras with MC twisted differentials it follows that the suboperad $\Tw^W\O\subset \Tw\O$ also acts on such objects. However, it is unclear whether any action of $\Tw^W\O$ can be extended to an action of $\Tw\O$; this indicates that, perhaps, the construction $\Tw$ is more `natural' than $\Tw^W$.
\end{enumerate}
\end{rem}


\begin{thebibliography}{99}
\bibitem{block} J. Block, \emph{Duality and equivalence of module categories in noncommutative geometry.} A celebration of the mathematical legacy of Raoul Bott, 311--339, CRM Proc. Lecture Notes, 50, Amer. Math. Soc., Providence, RI, 2010.
\bibitem{CL} J. Chuang, A. Lazarev,\emph{ L-infinity maps and twistings. }Homology, Homotopy and Applications, 13(2), (2011), 175--195, \texttt{arXiv:0912.1215}
%
\bibitem{Dol} V. Dolgushev, C. Rogers,
\emph{Notes on Algebraic Operads, Graph Complexes, and Willwacher's Construction. } \texttt{arXiv:math/0404003v4 }
\bibitem{Get} E. Getzler, \emph{Lie theory for nilpotent
L-infinity algebras}, Ann. Math., 170, no. 1 (2009), 271-301, \texttt{arXiv:math/0404003v4 }.
%
\bibitem{GeK} E.~Getzler, M.~M. Kapranov,
\emph{Cyclic operads and cyclic homology.}  Geometry, topology, \& physics,  167-201, Conf. Proc. Lecture Notes Geom. Topology, IV, Int. Press, Cambridge, MA, 1995.
%
\bibitem{GK} E.~Getzler, M.~M. Kapranov, \emph{Modular
    operads}. Compositio Math. 110 (1998), no.\ 1, 65--126.
%
\bibitem{GiK} V. Ginzburg, M.~Kapranov, \emph{Koszul duality for
    operads}. Duke Math. J. 76 (1994), no.\ 1, 203-272.

%
\bibitem{HL'} A. Hamilton,
 A. Lazarev, \emph{Homotopy algebras and noncommutative geometry}, \texttt{arXiv:math/0410621}.

%
\bibitem{Kon} M. Kontsevich, \emph{Formal Noncommutative Symplectic
Geometry.} The Gelfand Mathematical Seminars, 1990-1992, pp. 173-187,
Birkh\"auser Boston, Boston, MA, 1993.
%
\bibitem{Lef} S. Lefschetz, \emph{Algebraic Topology}, AMS, Colloquium Pbns. Series, Vol. 27, 1942.

\bibitem{Mer} S. Merkulov, \emph{An $L\sb \infty$-algebra of an
unobstructed deformation functor.} Internat. Math. Res. Notices 2000,
no. 3, 147-164.
\bibitem{Reu} C. Reutenauer,\emph{ Free Lie algebras}, London Mathematical Society Monographs. New Series, 7. Oxford Science Publications. The Clarendon Press, Oxford University Press, New York, 1993.
%
\bibitem{Sta} J. Stasheff, Jim, \emph{Deformation theory and the Batalin-Vilkovisky master equation.} Deformation theory and symplectic geometry (Ascona, 1996), 271--284, Math. Phys. Stud., 20, Kluwer Acad. Publ., Dordrecht, 1997.
%
\bibitem{Wil} T. Willwacher, \emph{M. Kontsevich's graph complex and the Grothendieck-Teichmueller Lie algebra}, \texttt{arxiv:1009.1654}.
%
\bibitem{wil} T. Willwacher, \emph{ A Note on Br-infinity and KS-infinity formality}, \texttt{arXiv:1109.3520}
    \end{thebibliography}
\end{document}